\documentclass[10pt, xcolor = {usenames,dvipsnames}]{article}

\usepackage[english]{babel}
\usepackage[utf8]{inputenc} 
\usepackage{amsfonts}
\usepackage[square,sort,comma,numbers]{natbib}
\bibpunct{[}{]}{;}{a}{,}{,}
\usepackage[hidelinks]{hyperref}

\usepackage{tikz}
\usetikzlibrary{calc}

\usepackage{mathtools}
\usepackage{amsthm}
\usepackage{amssymb}
\usepackage{amsmath}
\usepackage{needspace}
\usepackage{etoolbox}
\usepackage{lipsum}
\usepackage{authblk}
\usepackage{mathrsfs}
\usepackage{bbold}
\usepackage{graphicx}
\usepackage{subcaption}
\usepackage{color,soul}

\newcounter{theo}
\newtheorem{lmm}[theo]{Lemma}
\newtheorem{thr}{Theorem}
\newtheorem{stm}[theo]{Statement}
\newtheorem{crl}[theo]{Corollary}
\newtheorem{prp}[theo]{Proposition}

\newtheorem{defn}[theo]{Definition}
\newtheorem{rmrk}[theo]{Remark}

\newtheorem{thr_old}[theo]{Theorem}

\DeclareMathOperator{\arccosh}{arccosh}

\newcommand{\R}{\mathbb{R}}
\newcommand{\Compl}{\mathbb{C}}

\newcommand{\Z}{\mathbb{Z}}
\newcommand{\T}{\mathbb{T}}
\newcommand{\D}{\mathbb{D}}

\newcommand{\eps}{\varepsilon}

\newcommand{\Dyadic}[1]{D_{#1}}
\newcommand{\Aset}{\mathrm{A}}
\newcommand{\cmpct}{\mathrm{K}}
\newcommand{\Clos}{Cl}

\newcommand{\Diff}{\mathrm{Diff}}

\newcommand{\SL}{\mathrm{PSL}}

\newcommand{\PartF}{\mathcal{Z}}

\newcommand{\FMeas}[1]{\mathscr{M}_{#1}}
\newcommand{\FMeasSL}[1]{\widetilde{\mathscr{M}}_{#1}}

\newcommand{\FMInf}[1]{\mathscr{M}_{#1}^{\infty}}

\newcommand{\Emb}[1]{\mathsf{F}^{#1}} %
\newcommand{\GProj}{\mathsf{\Phi}} %
\newcommand{\EmbTan}[1]{\widetilde{\mathsf{F}}^{#1}} %

\newcommand{\Schw}{\mathcal{S}}

\newcommand{\A}{\mathsf{P}}
\newcommand{\B}{\mathsf{Q}}

\renewcommand*{\d}{\mathop{}\!\mathrm{d}}
\newcommand{\WS}[3]{\mathcal{B}_{#1}^{\, #2, #3}} %
\newcommand{\BM}[2]{\mathcal{W}_{#1}^{\, #2}} %

\newcommand{\Cfree}{C_{0,\text{free}}}
\newcommand{\CZPlus}{C^1_{0,+}}

\newcommand{\distCR}{\mathsf{d}_{C^1(\R)}}
\newcommand{\distContR}{\mathsf{d}_{C(\R)}}

\newcommand{\TopR}{\mathcal{T}_{C^1(\R)}}

\newcommand{\eventGlobal}{\mathrm{A}}
\newcommand{\evGlob}{\mathrm{B}}

\newcommand{\Dom}{\Omega}
\newcommand{\Real}{\Re}
\newcommand{\I}{\mathcal{I}}

\newcommand{\ObsHolLoc}[1]{\mathsf{LocH\ddot{o}lObs}^{\,#1}}

\makeatletter
\def\obs#1#2#3{\def\temp@expr{\big(#1; #2, #3\big)}\mathcal{O}\obs@}
\def\obs@{%
	\@ifnextchar{_}{\obs@sub}{
	\@ifnextchar{^}{\obs@sup}{\temp@expr}}}
\def\obs@sub#1#2{_{#2}\obs@}
\def\obs@sup#1#2{^{\, #2}\obs@}
\makeatother

\makeatletter
\def\cross#1#2#3#4#5{\def\temp@expr{\big(#1; #2, #3, #4, #5\big)}\widehat{\mathcal{O}}\cross@}
\def\cross@{%
	\@ifnextchar{_}{\cross@sub}{
	\@ifnextchar{^}{\cross@sup}{\temp@expr}}}
\def\cross@sub#1#2{_{#2}\cross@}
\def\cross@sup#1#2{^{\, #2}\cross@}
\makeatother

\newcommand{\fcolor}{blue}

\title{Schwarzian Field Theory at High Temperatures}
\author{Ilya Losev
\footnote{Mathematical Institute, University of Oxford, Andrew Wiles Building, Radcliffe Observatory Quarter, Woodstock Road, Oxford, OX2 6GG, UK.
E-mail: \url{ilya.losev@maths.ox.ac.uk}.}
}

\date{July 31, 2026}

\begin{document}
\maketitle

\begin{abstract}
In this work we study the high-temperature limit of the Schwarzian Field Theory probability measure.
We show that on large scales this limit concentrates on jump processes, while its small-scale behaviour is governed by a process which we call Schwarzian Field Theory on the real line.
This Schwarzian Theory on the real line can be viewed as the infinite-volume version of the Schwarzian Field Theory.
In addition to showing this local convergence, we also provide a systematic treatment of the Schwarzian Theory on the real line. 
This includes the calculation of the correlation functions and a corresponding uniqueness theorem.
\end{abstract}

\maketitle

\begingroup
\renewcommand\thefootnote{}
\footnotetext{
\textbf{AI disclosure:}
In this work, AI tools were used solely for proofreading and linguistic refinement.
}
\addtocounter{footnote}{-1}
\endgroup

\section{Introduction and main results}

\subsection{Introduction}
Schwarzian Field Theory is a quantum field theory that has recently attracted a lot of attention in physics in the context of AdS/CFT correspondence and black holes.
It is predicted to arise as a holographic dual of the Jackiw-Teitelboim (JT) gravity in the disk \citep{SaadShenkerStanford2019, NearlyAdS, JT_Wilson_Line, Ferrari_Lattice_JT}.
The Schwarzian Field Theory also emerges in the low-energy limit of the Sachdev–Ye–Kitaev (SYK) random matrix model (e.g. see \citep{MaldacenaStanford, KitaevJosephine}), and has connections to Liouville CFT \citep{ConformalBootstrap}, infinite-dimensional symplectic geometry \citep{AlekseevBosonization, AlekseevShatashvili, StanfordWittenFermionicLocalization}, representation theory of the Virasoro algebra \citep{AlekseevShatashvili2}, 2D Yang-Mills \citep{Schwarzian_Wilson_Line}, and random polygons \citep{chekhov2024, chekhov_budd}.

In \citep*{BLW}, we defined a finite measure on
$\mathrm{Diff}^{1}(\mathbb{T})/\mathrm{PSL}(2, \mathbb{R})$ which corresponds
to the Schwarzian Field Theory, and also calculated its partition function (i.e. total mass).
We call this measure \emph{Schwarzian Measure on the circle}.
A convincing outline of the construction of the Schwarzian measure in terms of a Brownian bridge was given by \citep{BelokurovShavgulidzeExactSolutionSchwarz, BelokurovShavgulidzeCorrelationFunctionsSchwarz}. In \citep{BLW,LosevCorr}, we charcterized the Schwarzian measure uniquely in terms of a change of variable formula and in terms of its correlation functions. This measure agrees with a rigorous version of the Brownian Bridge construction.

\medskip

One of the questions that naturally arises after one constructs a measure corresponding to a quantum field theory, is to describe its behaviour at both extreme temperatures.
When it comes to the Schwarzian Field Theory, these questions also have a physical significance.
For example, in the appropriate large-boundary limit of JT gravity, it is predicted that the Schwarzian Field Theory emerges with a temperature that depends on the relationship between parameters of the JT gravity model (these parameters are either boundary length and dilaton boundary condition in the notation of \citep{MaldacenaStanford, KitaevJosephine} or boundary length and cosmological constant in \citep{Ferrari_Lattice_JT}).
Therefore, one naturally expects to see both low-temperature and high-temperature (or even infinite-temperature) Schwarzian Field Theory arise from the JT gravity model in certain regimes.

Furthermore, we conjecture that the high-temperature limit of Schwarzian Field Theory should naturally emerge in a broader class of models.
Indeed, the Schwarzian Field Theory is expected to arise universally in central limit theorems in various settings with an emergent conformal symmetry.
In these situations, the Schwarzian Field Theory action appears as a variation of the underlying model's action.
In other words, the Schwarzian action should emerge with an infinitesimally small parameter in front of it (see, e.g.,~\citep{KitaevJosephine, Yilin_schwarzian_action}). 
Consequently, it is specifically the infinite-temperature limit of the Schwarzian Field Theory that should govern these central limit behaviours.

Alternatively, we can also view the high-temperature limit of Schwarzian Field Theory as its long-time limit.
Indeed, we can absorb the temperature parameter into the time parameter by rescaling the argument of the field.
This perspective puts the question of studying the high-temperature/long-time limit of the Schwarzian Theory into a broader program in the area of constructive field theory, with the goal of understanding probabilistic QFTs in their infinite-volume limits, see \citep{GubHairerOhZine2025, DuchGubinelliRinaldi2025, GubinelliHofmanova, DuchP2025, BarashkovGubinelli2021, ShenZhuZhu} for some of the recent progress.

\medskip

In this work we study the high-temperature limit of the Schwarzian Measure on the circle, as defined in \citep*{BLW}.
We show that on small scales this limit is described by a process that we call \emph{Schwarzian Measure on the real line}.
This process is essentially an infinite-time version of the Schwarzian Field Theory.
In addition to defining this process and proving the local convergence, we also calculate its correlation functions of cross-ratio observables whenever the corresponding Wilson lines are non-intersecting, and prove that these correlation functions determine this process uniquely.
We also show that on large scales the high-temperature Schwarzian Field Theory concentrates on jump processes:
with high probability, every representative has at most three macroscopic jumps, and the precise number and placement of these jumps depend on the choice of representative in the $\SL(2, \R)$-orbit.

This paper complements \citep{LosevLDP}, where we investigated the low-temperature regime of the Schwarzian Measure on the circle, proving the large deviation principle with the good rate function given by the action of the theory.

\subsubsection{Probabilistic setting}

\emph{Schwarzian Measure on the circle} is the measure corresponding to the Schwarzian Field Theory. 
It is supported on the topological space $\Diff^1(\T)/\SL(2, \R)$ and is formally given by (see \citep[(1.1)]{StanfordWittenFermionicLocalization})
\begin{equation}\label{eq:1}
\d\FMeas{\sigma^2}\big(\phi\big) = 
\exp\left\{- \frac{1}{\sigma^2} \, \I(\phi)\right\}
\frac{\prod_{\tau \in \T}\frac{\d\phi(\tau)}{\phi'(\tau)}}{\SL(2, \R)} ,
\end{equation} 
where the action $\I(\phi)$ is
\begin{equation}
\I(\phi)
=
-\int_{\T} 
\left[\Schw(\phi, \tau)+2\pi^2\phi'^{\, 2}(\tau)\right] \d\tau
=
-\int_{\T} \Schw\big(e^{2\pi i \phi}, \tau \big)\d\tau,
\end{equation}
and $\Schw_{\phi}(\tau)$ is the Schwarzian derivative of $\phi$,
\begin{equation}
\Schw(\phi, \tau) =
\Schw_{\phi}(\tau)
=\left(\frac{\phi''(\tau)}{\phi'(\tau)}\right)'- \frac{1}{2} \left(\frac{\phi''(\tau)}{\phi'(\tau)}\right)^2.
\end{equation}
Here, $\T = [0, 1]/\{0\sim 1\}$ is the unit circle, $\Diff^1(\T)$ is the space of $C^1$ orientation preserving diffeomorphisms of $\T$, and $\SL(2, \R)$ is the group of M\"{o}bius transformations of the unit disk (i.e. conformal isomorphisms of the unit disk) restricted to the boundary which is identified with $\T$.
The group $\SL(2, \R)$ acts on $\Diff^1(\T)$ by post-compositions. 
Following \citep{StanfordWittenFermionicLocalization} we call it a right action, since in \citep{StanfordWittenFermionicLocalization} it is interpreted as an action on the inverse elements.
We denote the quotient of $\Diff^1(\T)$ by this action of $\SL(2, \R)$ by $\Diff^1(\T)/\SL(2, \R)$.
Heuristically, the formal density~\eqref{eq:1} only depends on the orbit of this action and the quotient by $\SL(2,\R)$, therefore, makes sense.
We recall the rigorous construction of the Schwarzian Measure on the circle carried out in \citep*{BLW} in Section~\ref{sect_meas_construct}.

The partition function of the Schwarzian Field Theory (i.e. total mass of the corresponding measure) was originally formally derived in \citep{StanfordWittenFermionicLocalization} and later rigorously obtained in \citep*{BLW}, and is given by
\begin{equation}\label{eqPartF_intro}
\PartF\!\left(\sigma^2\right) = 
\mathcal{Z}_{\sigma^2}=
\left(\frac{2\pi}{\sigma^2}\right)^{3/2} \exp\left(\frac{2\pi^2}{\sigma^2}\right)
    =\int_0^{\infty} e^{-{\sigma^2k^2}/{2}} \sinh(2\pi k) \, 2 k \d k.
\end{equation}

The parameter $\sigma^2$ in \eqref{eq:1} corresponds to the \emph{temperature} of the system. 
In this paper we study the $\sigma^2 \to \infty$ limit of the normalised Schwarzian Measures $\mathcal{Z}_{\sigma^2}^{-1} \cdot \d\FMeas{\sigma^2}(\phi)$.

In Theorem~\ref{thrLocalLimit} we determine the local structure of this limit.
We prove that after fixing the gauge (choosing a representative in each conjugacy class) and appropriate rescaling of $\phi$ these normalised Schwarzian Measures on the circle converge to a random process, which can be defined by
\begin{equation}
t \mapsto \int_{0}^t e^{\xi(s)}\d s,
\end{equation}
where $\xi$ is the two-sided Brownian Motion on the real line (see Definition~\ref{defnSchwLine}).
We call this stochastic process \emph{Schwarzian Measure on the real line}.
In Section~\ref{sectSchwLineIntro} we provide a systematic treatment of this process, calculating its correlation functions  for non-interlaced observables, proving a uniqueness theorem, and developing H\"{o}lder property theory for it and its observables.
Even though results from Section~\ref{sectSchwLineIntro} are of independent interest, here we also heavily rely on them to prove the convergence stated in Theorem~\ref{thrLocalLimit}.

In addition, we also show that without rescaling the $\sigma\to \infty$  limit of $\mathcal{Z}_{\sigma^2}^{-1} \cdot \d\FMeas{\sigma^2}(\phi)$ concentrates on jump processes.
Moreover, we show that these processes have no more than $3$ jumps, and that the exact number (and structure) of jumps heavily depends on the choice of representatives in conjugacy classes.

\subsection{Main results}
In this section we state the main results concerning the high-temperature limit of the Schwarzian Measure on the circle.

\subsubsection{Local structure}
First, we describe the local structure of diffeomorphisms which are sampled according to the high-temperature Schwarzian Field Theory.

In order to formulate the result, we need to fix the gauge (choose a representative in each conjugacy class) and describe the rescaling. 
Let $\GProj (\phi) \coloneqq \GProj_\phi$ be the representative of the equivalence class of $\phi\in \Diff^1(\T)/\SL(2, \R)$, with
\begin{equation}
\GProj_\phi(0) = 0,
\qquad \GProj'_\phi(0) = 1,
\qquad \GProj'_\phi(1/2) = 1/2.
\end{equation}
Further, we embed $\Diff^1(\T)/\SL(2, \R)$ into $C^1(\R) $ by rescaling both the argument and the value by the factor of $\beta$, and using the periodic lift,
\begin{align}
\Emb{\beta}:\, \Diff^1(\T)/\SL(2, \R) &\to C^1(\R)  \\
\phi &\mapsto \Emb{\beta}(\phi)(t) \coloneqq \Emb{\beta}_\phi (t)
\coloneqq \beta \Big(\lfloor \tfrac{t}{\beta} \rfloor + \GProj_\phi\big(\tfrac{t}{\beta}\, \mathrm{mod}\, 1\big)\Big).
\label{eqDefFMap}
\end{align}

We endow $C^1(\R)$ with topology $\TopR$,
which is induced by the metric 
\begin{equation}\label{eqDefMetricC1}
\distCR (f, g) = \sum_{n=1}^{\infty} 2^{-n} \, \min\left\{1, \, \|f-g\|_{C^1 [-n, n]} \right\},
\end{equation}
where
\begin{equation}
\|f-g\|_{C^1 [-n, n]}
=
\max \left\{
\sup_{x\in [-n, n]} |f(x) - g(x)|, \,
\sup_{x\in [-n, n]} |f'(x) - g'(x)|
\right\}.
\end{equation}
In other words, topology $\TopR$ encodes uniform $C^1$ convergence on compact sets.
To be more precise, $f_n\to g $ in $\TopR$ if and only if for every compact $K\subset \R$, we have that both $f_n \to g$ and $f_n'\to g'$ uniformly on $K$.
Moreover, as for functional spaces on finite intervals, the Borel $\sigma$-algebra of $\TopR$ coincides with the $\sigma$-algebra generated by cylinder sets of functions and their derivatives. 
Also note that $C^1(\R)$ with topology $\TopR$ is a Polish space.

\begin{rmrk}
Similar generalisation of metrics (and topologies) on functional spaces on finite intervals to metrics on functional spaces on infinite intervals is well-known in the context of Skorokhod theory of c\`{a}dl\`{a}g functions $D[0, \infty)$, see~\citep[Section~16]{Billingsley}.
Such topologies, which encode convergence on compact sets, are natural for stochastic processes.
\end{rmrk}

Let $\d\Emb{\beta}_{\sharp}\FMeas{\sigma^2}$ be the push-forward of the Schwarzian Measure on the circle (see Section~\ref{sect_meas_construct} for definitions and main properties) under \eqref{eqDefFMap}. 
Since \eqref{eqDefFMap} is continuous, we get that $\d\Emb{\beta}_{\sharp}\FMeas{\sigma^2}$ is a Borel finite measure on~$\big(C^1(\R),  \TopR\big)$.
Recall that its total mass $\PartF_{\sigma^2}$ is given by \eqref{eqPartF_intro}.
Our goal is to show that if $\sigma, \beta\to\infty$ with $\sigma^2/\beta \to \varkappa^2$ for some $\varkappa>0$, then the normalised push-forwards $\PartF_{\sigma^2}^{-1}\cdot \d\Emb{\beta}_{\sharp}\FMeas{\sigma^2}$ converge to a probability measure, which we call Schwarzian Measure on the real line.

\begin{defn}\label{defnSchwLine}
Let $\varkappa>0$.
\emph{Schwarzian Measure on the real line} is the Borel probability measure $\d\FMInf{\varkappa^2}$ on~$\big(C^1(\R),  \TopR\big)$ given by
\begin{equation}\label{eqDefLineSchw}
\d\FMInf{\varkappa^2}(f)
 \coloneqq 
\d \BM{\varkappa^2}{\infty}(\xi), 
\qquad \text{with } f(t) = \B_{\xi}(t),
\end{equation}
where 
\begin{equation}\label{eqDefQ}
\B(\xi)(t):= \B_{\xi}(t):= \int_{0}^{t} e^{\xi(s)}\d s,
\end{equation}
and $\d\BM{\varkappa^2}{\infty}(\xi)$ is a probability measure on $C(\R)$, corresponding to the two-sided Brownian Motion on $\R$ with $\xi(0)=0$ and variance $\varkappa^2>0$.
\end{defn}

\begin{rmrk}
Here we endow $C(\R)$ with the topology of locally uniform convergence, generated by the metric
\begin{equation}
\distContR (f, g) = \sum_{n=1}^{\infty} 2^{-n} \, \min\left\{1, \,\sup_{x\in [-n, n]} |f(x) - g(x)|\right\}.
\end{equation}
We have that the Borel $\sigma$-algebra generated by this topology coincides with the $\sigma$-algebra generated by cylinder sets of function values (see~\citep[Section~16]{Billingsley}). 
In particular, two-sided Brownian Motion $\d\BM{\varkappa^2}{\infty}$ is defined as a probability measure on the Borel $\sigma$-algebra of $C(\R)$.
Moreover, with our choice of topologies we have that the map $\B$, given by \eqref{eqDefQ}, is continuous as a map from $C(\R)$ to $C^1(\R)$.
\end{rmrk}
Formally,
\begin{equation}
\d\FMInf{\varkappa^2}(f)
=
\exp\left\{ \frac{1}{\varkappa^2} \, 
\int_{\R} \Schw\big(f, \tau \big)
\right\}
\prod_{\tau \in \R}\frac{\d f(\tau)}{f'(\tau)}.
\end{equation}
Thus, we can view this measure as the infinite-time version of the Schwarzian Field Theory.
We expect that the Schwarzian Measure on the real line defined above should arise as a limit in various physical models, including the JT gravity. 
We further discuss the Schwarzian Measure on the real line $\d\FMInf{\varkappa^2}$ and its properties in Section~\ref{sectSchwLineIntro}.

The main result of this section is the following theorem.

\begin{thr}\label{thrLocalLimit}
If $\sigma, \beta \to\infty$ with $\sigma^2/\beta \to \varkappa^2$, then the normalised push-forwards $\PartF_{\sigma^2}^{-1}\cdot \d\Emb{\beta}_{\sharp}\FMeas{\sigma^2}$ converge weakly to $\d\FMInf{\varkappa^2}$ as Borel measures on the topological space $(C^1(\R), \TopR)$.
\end{thr}
We discuss the proof strategy in Section~\ref{sectProofStrat} and provide the proof of this theorem in Section~\ref{sectThrLocalProof}.

\subsubsection{Global structure}
We can also show that when the temperature $\sigma^2$ is large, the Schwarzian Measure on the circle $\d\FMeas{\sigma^2}$ concentrates on such classes of diffeomorphisms $[\phi] \in \Diff^1(\T)/\SL(2, \R)$ that all their representatives are jump processes with at most $3$ jumps.
\begin{thr}\label{thrGlobalConvJumpBound}
Let $0=t_1<t_2<\ldots<t_N< t_{N+1} = 1$ be distinct points on the circle.
For any $\eps>0$ and any integer $N>0$ there exists a Borel measurable set $\eventGlobal_{N, \eps} \subset \Diff^1(\T)/\SL(2, \R)$ such that 
\begin{enumerate}
\item We have
\begin{equation}
\lim_{\sigma\to \infty} \mathcal{Z}^{-1}_{\sigma^2} \cdot \FMeas{\sigma^2} \left(
\eventGlobal_{N, \eps}
\right) = 1.
\end{equation}

\item
For any $\phi\in \Diff^1(\T)$ which is a representative of some conjugacy class $[\phi]\in \eventGlobal_{N, \eps}$ we have that
\begin{equation}
\left|
\left\{
j\in\{1, 2, \ldots N \}: \,
\phi(t_{j+1}) - \phi(t_{j}) > \eps
\right\}
\right| \leq 3.
\end{equation}
\end{enumerate}

\end{thr}

Moreover, the exact number of jumps depends on the representative, and with high probability can be any number between $1$ and $3$.
Below we show that for any $k\in\{1, 2, 3\}$ with large probability (in $\sigma\to \infty$ regime) we can find a representative that has $k$ jumps of size around $1/k$.
We also ensure that these jumps are separated by a given $\eta$.
The introduction of this $\eta$ ensures that the jumps, indeed, happen at different points and do not merge in the limit.

\begin{thr}\label{thrGlobalConvJumpPossible}
For any $k \in \{1, 2, 3\}$ and $\eta>0$ there exists $\rho(\eta)>0$ such that
for any $\eps>0$ we can find
an event $\evGlob_{k, \eta, \eps} \subset \Diff^1(\T)/\SL(2, \R)$ such that the following holds.
\begin{enumerate}
\item We have that
\begin{equation}
\liminf_{\sigma\to \infty} \mathcal{Z}^{-1}_{\sigma^2} \cdot \FMeas{\sigma^2} \left(
\evGlob_{k, \eta, \eps}
\right) 
> 1-\rho(\eta),
\end{equation}
and $\lim_{\eta\to 0}\rho(\eta) = 0$.

\item
For any conjugacy class $[\phi]\in \evGlob_{k, \eta,  \eps}$ there exist a representative $\phi\in \Diff^1(\T)$ and
$\{t_j\}_{j=1}^k \subset \T$ with
\begin{equation}
\forall i \neq j \in \{1, \ldots, k\}: \qquad |t_j-t_i| > \eta,
\end{equation}
such that 
\begin{equation}
\forall j\in \{1, \ldots, k\}: \qquad
\phi(t_{j}+\eps) - \phi(t_{j}-\eps) > \frac{1}{k} - \eps.
\end{equation}
\end{enumerate}
\end{thr}

This theorem also implies that the global limit of $\lim_{\sigma\to \infty} \mathcal{Z}^{-1}_{\sigma^2} \cdot \FMeas{\sigma^2} $ crucially depends on the choice of the representatives.

We explain the main ideas behind proofs of both Theorem~\ref{thrGlobalConvJumpBound} and Theorem~\ref{thrGlobalConvJumpPossible} in Section~\ref{sectProofStrat}, and prove them in Section~\ref{sectThrGlobProof}.

\subsection{Schwarzian Measure on the real line}
\label{sectSchwLineIntro}

In this section we state the main properties of Schwarzian Measure on the real line $\d\FMInf{\varkappa^2}$, which was introduced in Definition~\ref{defnSchwLine} as the push-forward of the Brownian Motion measure under \eqref{eqDefQ}.

Note that push-forwards of the Brownian Bridge measures under \eqref{eqDefQ} have already appeared in \citep{LosevCorr}, where they were used to calculate correlation functions of the Schwarzian Measure, as defined in~\citep*{BLW}.
Here we will call them Schwarzian Measures on intervals.
We recall the relevant facts in Section~\ref{sectSchwInterval}.
However, probabilistic Schwarzian Measure on the real line has not appeared in previous work.
Below we give a systematic treatment of Schwarzian Measure on the real line.
Note that, even though many results are similar to those for Schwarzian Measure on the circle, as derived in \citep{LosevCorr, LosevLDP}, some of the proofs require significant adaptations.
The main results from this section are proved in Section~\ref{sectSchwLineProofs}.
In this section we provide proofs for those results, which follow immediately either from \citep{LosevCorr, LosevLDP} (we recall the necessary results in Section~\ref{sect_meas_construct}) or from other results in this section.

\subsubsection{Correlation functions}

We start with the results concerning the correlation functions of cross-ratio observables for the measure $\d\FMInf{\varkappa^2}$.
The relevant observables are
\begin{equation}\label{eqObs_0_def}
\obs{f}{s}{t}_0
=
\frac{\sqrt{f'(s)f'(t)}}{|f(t)-f(s)|}, 
\end{equation}
for $s \neq t\in \R$ and monotonic $f\in C^1(\R)$. 
To be more precise, in this section and throughout the paper we will work with the subset $\CZPlus(\R)$ of $C^1(\R)$ 
\begin{equation}
\CZPlus(\R) = \left\{f\in C^1(\R) \left| \,\,
f(0) =0; \, f'(0)=1; \, \forall t\in \R: \, f'(t)>0
\right.\right\}.
\end{equation}
The set $\CZPlus(\R)$ inherits the topology $\TopR$ from $C^1(\R)$.
From Definition~\ref{defnSchwLine} it is clear that the measure $\d\FMInf{\varkappa^2}$ is supported on $\CZPlus(\R)$.

\begin{rmrk}\label{rmrkSLLineAction}
There is an alternative way to think about $\d\FMInf{\varkappa^2}$ and its support.
Consider the set $C^1\big(\R \to \R\cup\{\infty\}\big)$ of functions from $\R$ to $\R\cup\{\infty\}$, identifying the latter with the circle.
Then $\SL(2, \R)$ acts on  $C^1\big(\R \to \R\cup\{\infty\}\big)$ by post-compositions with fractional linear transformations
\begin{equation}
f  \mapsto 
\frac{af+b}{cf+d}, \qquad \text{ for } a, b, c, d\in \R, \, ad-bc = 1.
\end{equation}
From this perspective we can identify $\CZPlus \cong C^1\big(\R \to \R\cup\{\infty\}\big)/\SL(2, \R)$.
In other words, one gets $\CZPlus$ from $C^1\big(\R \to \R\cup\{\infty\}\big)/\SL(2, \R)$ by fixing the gauge $f(0)=0, f'(0)=1, f(\infty) = \infty$.

Thus, we can view $\d\FMInf{\varkappa^2}$ as being supported on the quotient space $C^1\big(\R \to \R\cup\{\infty\}\big)/\SL(2, \R)$.
Moreover, observables $\obs{f}{\cdot}{\cdot}_0$, introduced in \eqref{eqObs_0_def}, are invariant under this $\SL(2, \R)$ action and, thus, define observables on $C^1\big(\R \to \R\cup\{\infty\}\big)/\SL(2, \R)$. 
This point of view is similar to our approach to the Schwarzian Measure on the circle, which we think of being supported on the quotient space $\Diff^1(\T)/\SL(2, \R)$.
In the case of Schwarzian Measure on the line, however, we prefer to work with the explicit gauge (i.e. with the space $\CZPlus(\R)$), as it seems more natural.

Note also that the action of $\SL(2, \R)$ on $\Diff^1(\T)$ is different from its action on $C^1\big(\R \to \R\cup\{\infty\}\big)/\SL(2, \R)$ described here.
We further discuss how these two actions are related in Section~\ref{sectProofStrat}.
\end{rmrk}

\medskip

Here we describe a diagrammatic representation for observables $\obs{f}{\cdot}{\cdot}_0$ defined on the real line $\R$. 
This representation is similar to the one described for the circle in \citep{LosevCorr}. 

Let $\left\{s_j\right\}_{j=1}^N$  and $\left\{t_j\right\}_{j=1}^N$ be points on $\R$, such that $\forall j: s_j < t_j$.
Let also $\Big\{\obs{f}{s_{j}}{t_j}_0\Big\}_{j=1}^{N}$ be a set of observables. 
We represent them as a diagram on the real line.

\begin{figure}[tb]\centering
\begin{subfigure}[t]{0.55\textwidth}
\centering
\begin{tikzpicture}
\newcommand\radius{8}
\newcommand\ra{0.43}
\newcommand\ca{0.5}
\newcommand\rb{0.19}
\newcommand\cb{0.39}
\newcommand\rc{0.1}
\newcommand\cc{0.75}
\newcommand\rd{0.1}
\newcommand\cd{0.3}

	\node (n0) at (0, 0) {};
	\node (nT) at (\radius, 0) {};
	\node[shape=circle, fill=black, scale=0.5,label={-90:$s_1$}] (ns1) at (\ca *\radius - \ra *\radius, 0) {};
	\node[shape=circle, fill=black, scale=0.5,label={-90:$t_1$}] (nt1) at (\ca *\radius + \ra *\radius, 0) {};
	\node[shape=circle, fill=black, scale=0.5,label={-90:$s_2=s_4$}] (ns2) at (\cb *\radius - \rb *\radius, 0) {};
	\node[shape=circle, fill=black, scale=0.5,label={-90:$t_2$}] (nt2) at (\cb *\radius + \rb *\radius, 0) {};
	\node[shape=circle, fill=black, scale=0.5,label={-90:$s_3$}] (ns3) at (\cc *\radius - \rc *\radius, 0) {};
	\node[shape=circle, fill=black, scale=0.5,label={-90:$t_3$}] (nt3) at (\cc *\radius + \rc *\radius, 0) {};
	\node[shape=circle, fill=black, scale=0.5,label={-90:$t_4$}] (nt4) at (\cd *\radius + \rd *\radius, 0) {};

  	\draw (n0) -- (nT);
  	\draw (\ca *\radius + \ra *\radius, 0) arc  (0:180:\ra *\radius);
  	\draw (\cb *\radius + \rb *\radius, 0) arc  (0:180:\rb *\radius);
  	\draw (\cc *\radius + \rc *\radius, 0) arc  (0:180:\rc *\radius);
  	\draw (\cd *\radius + \rd *\radius, 0) arc  (0:180:\rd *\radius);

  \node[label={90:$\color{\fcolor} k_0$}] (nf0) at ( 0.5* \radius , 0.45* \radius) {};
  \node[label={90:$\color{\fcolor} k_1$}] (nf1) at ( 0.55* \radius , 0.23* \radius) {};
  \node[label={90:$\color{\fcolor} k_2$}] (nf2) at ( 0.44* \radius , 0.05* \radius) {};
  \node[label={90:$\color{\fcolor} k_3$}] (nf3) at ( 0.3* \radius , 0.0* \radius) {};
  \node[label={90:$\color{\fcolor} k_4$}] (nf4) at ( 0.75* \radius , 0.0* \radius) {};

\end{tikzpicture}
\caption{An example of a diagram. Here,\\ 
$\tau_0 = \infty$, \\
$\tau_1 = (s_2-s_1)+(s_3-t_2)+(t_1-t_3)$, \\
$\tau_2 = t_2-t_4$, \\
$\tau_3 = t_4-s_4$,\\ 
$\tau_4 = t_3-s_3$.
}\label{fig:observables_line_diagram_def_fourier}
\end{subfigure}
\hspace{0.05\textwidth}
\begin{subfigure}[t]{0.36\textwidth}
\centering
\begin{tikzpicture}
\newcommand\radius{4}
 	\draw (-0.2*\radius, 0) -- (0.2*\radius, 0);
 	\draw (0.8* \radius, 0) -- (1.2*\radius, 0);
  \node[shape=circle,fill=black, scale=0.5,label={-90:$s_j$}] (ns1) at (0, 0) {};
  \node[shape=circle,fill=black, scale=0.5,label={-90:$t_j$}] (nt1) at (\radius, 0) {};
 
	\draw (0,0) arc (180: 0:0.5*\radius);

	\node[label={90:$\color{\fcolor} w_2(j)$}] (nf2) at (0.5*\radius, 0.2*\radius) {};
	\node[label={90:$\color{\fcolor} w_1(j)$}] (nf1) at (0.5*\radius, 0.55*\radius) {};

\end{tikzpicture}
\caption{We use $w_1(j)$ and $w_2(j)$ to denote the Fourier variables $k_m$ corresponding to the domains that lie on both sides of the half-circle that connects $s_j$ with $t_j$. \\
In Figure \ref{fig:observables_line_diagram_def_fourier}, for example, $w_1(2) = k_1$ and $w_2(2) = k_2$ (or vice versa).}
\label{fig:observables_line_diagram_obs_fourier}
\end{subfigure}
\caption{A diagrammatic representation of Fourier variables.}
\end{figure}

We draw the real line $\R$ and all the points $\{s_j\}_{j=1}^N$ and $\{t_j\}_{j=1}^N$ on it. 
For all $1\leq j \leq N$ we connect the point $s_j$ with the point $t_j$ with a half-circle in the upper half-plane.

\begin{defn}[\citep{LosevCorr}]
We say that a set of observables $\Big\{\obs{f}{s_{j}}{t_j}_0\Big\}_{j=1}^{N}$ is non-interlaced if the interiors of all the drawn half-circles on the corresponding diagram are pairwise non-intersecting. In other words, $\forall p, q \in \{1, \ldots, N\}$ we have that one of the following holds: $s_p\leq s_q < t_q\leq t_p$, $s_p< t_p\leq s_q< t_q$, $s_q\leq s_p < t_p\leq t_q$, or $s_q< t_q\leq s_p< t_p$.
\end{defn}

Let $\Big\{\obs{f}{s_{j}}{t_j}_0\Big\}_{j=1}^{N}$ be a set of non-interlaced observables. 
The $N$ drawn half-circles on the corresponding diagram divide the half-plane into $N+1$ connected domains, one of which is unbounded and the $N$ others are bounded.
We number the bounded domains with integers from $1$ to $N$, and assign number $0$ to the unbounded domain.
For each $m \in \{0, \ldots N\}$ we associate a Fourier variable $k_m$ to domain number $m$. 
We also let $\tau_m$ be the total length of all line segments (parts of $\R$) which form the boundary of the $m$-th domain (see Figure \ref{fig:observables_line_diagram_def_fourier}).

For each $j \in \{1, \ldots N\}$ we define $w_1(j)$ and $w_2(j)$ to be the Fourier variables corresponding to the domains that contain the half-circle connecting $s_j$ with $t_j$ in their boundaries (see Figure \ref{fig:observables_line_diagram_obs_fourier}). 
All formulae will be symmetric in $w_1$ and $w_2$, so the exact order does not matter.

We use the notation $\Gamma(l\pm ik\pm iw)$ for
\begin{equation}
\Gamma(l\pm ik\pm iw) := 
\Gamma(l + ik + iw)  \Gamma(l + ik - iw)  \Gamma(l - ik + iw)  \Gamma(l - ik - iw).
\end{equation}
\begin{thr} \label{thrLineCorrelations}
For $N\geq 0$ let $\big\{\obs{f}{s_{j}}{t_j}_0\big\}_{j=1}^{N}$ be a set of non-interlaced observables, and $\{l_j\}_{j=1}^N$ be a set of positive integers. 
Let also $\{k_m\}_{m=0}^{N}$ and $\{\tau_m\}_{m=0}^{N}$ be as above.
Then 
\begin{multline}
\int\limits_{\CZPlus(\R)} \prod_{j=1}^N \obs{f}{s_j}{t_j}^{l_j}_0\d\FMInf{\varkappa^2}(f)
=
\int_{\R_+^{N}} 
\prod_{j=1}^N \frac{\Gamma\Big(\frac{l_j}{2}\pm i w_1(j) \pm i w_2(j)\Big) }{2 \pi^2\, \Gamma(l_j)} \cdot \left(\frac{\varkappa^2}{2}\right)^{l_j}\\
\times
\left.
\prod_{m=1}^{N}\exp\left(-\frac{\tau_m \varkappa^2}{2}\cdot k_m^2\right)  \sinh(2\pi k_m)\, 2 k_m \d k_m\right|_{k_0 = 0},
\end{multline}
where the integral on the right-hand side converges absolutely.
\end{thr}
\begin{proof}
This follows from Proposition~\ref{prpIntervalExpObsFormula_corr} after integration over $a$.
\end{proof}

\begin{crl}\label{crlExpMomentLine}
For any $\varkappa>0$ and any $s\neq t \in \R$, 
\begin{equation}
\int_{\CZPlus(\R)}  \exp\left\{ \frac{8}{\varkappa^2}\, \obs{f}{s}{t}_0 \right\} \d\FMInf{\varkappa^2}(f) < \infty.
\end{equation}
\end{crl}

\begin{rmrk}\label{rmrkExpOptLine}
The constant $8/\varkappa^2$ in the exponential is sharp.
\end{rmrk}

\begin{proof}
This follows from Proposition~\ref{prpObsExpMomentBound_corr} after integration over $a$.
\end{proof}

Probability measure $\d\FMInf{\varkappa^2}$ on $\CZPlus(\R)$ is characterised by the correlation functions computed in Theorem~\ref{thrLineCorrelations}.
\begin{thr}\label{thrUniqLineSchw}
Let $\mathcal{P}$ be a Borel measure on $\CZPlus(\R)$. 
Suppose that for any $N\geq 0$, any set $\big\{\obs{f}{s_{j}}{t_j}_0\big\}_{j=1}^{N}$ of non-interlaced observables, and any set $\{l_j\}_{j=1}^N$ of positive integers we have 
\begin{equation}\label{eqThrUniqCorrAssumpt}
\int \prod_{j=1}^N \obs{f}{s_j}{t_j}^{l_j}_0 \d\mathcal{P}(f)
=
\int \prod_{j=1}^N \obs{f}{s_j}{t_j}^{l_j}_0 \d\FMInf{\varkappa^2}(f).
\end{equation}
Then for any Borel set $\Aset\subset \CZPlus(\R)$ we have
\begin{equation}
\mathcal{P}(\Aset) = \FMInf{\varkappa^2}(\Aset).
\end{equation}
\end{thr}
\begin{rmrk}
By taking $N=0$ in \eqref{eqThrUniqCorrAssumpt} we see that the total mass of $\d\mathcal{P}$ equals the total mass of $\d\FMInf{\varkappa^2}$. 
Thus, $\d\mathcal{P}$ is also a probability measure.
\end{rmrk}
We prove Theorem~\ref{thrUniqLineSchw} in Section~\ref{sectSchwLineUnique}.
We use this uniqueness theorem to characterise the limit in the proof of Theorem~\ref{thrLocalLimit}.

\subsubsection{H\"{o}lder condition for observables}
It is also possible to relate observables $\obs{f}{\cdot}{\cdot}_0$ to the H\"{o}lder property.
This observation and related techniques are key for proving tightness of measures in the context of Theorem~\ref{thrLocalLimit}.
\begin{defn}\label{defnHolderLine}
We say that $f\in \CZPlus(\R)$ satisfies \emph{Local H\"{o}lder condition for observables} at $x\in\R$ with exponent $\alpha$ and constant $A$ if for some $\delta>0$,
\begin{equation} \label{eqObsHolderLine}
\left|\obs{f}{s}{t}_0 - \frac{1}{|t-s|}\right| \leq A \left| t-s \right|^{\alpha-1}
\end{equation}
for all $s \neq t \in (x-\delta, x+\delta)$.
We also write $f\in \ObsHolLoc{\alpha}(x, \delta, A)\subset \CZPlus(\R)$.
\end{defn}

We show that the Local H\"{o}lder condition for observables, defined above, is equivalent to the local H\"{o}lder condition for the logarithm of the derivative.

\begin{thr}\label{thrLocObsHolderLine}
Fix $\alpha \in (0, 1)$.
For any $A, \delta > 0$ there exist $A'$ and $\delta'$ such that
if for some $x\in\R$ we have $\left|\log f'(t) - \log f'(s) \right| \leq A|t-s|^{\alpha}$ for all $s, t\in (x-\delta, x+\delta)$, then
$f\in \ObsHolLoc{\alpha}(x, \delta', A')$.

A converse is also true. 
For any $A', \delta' > 0$ there exist $A''$ and $\delta''$ such that
if $f\in \ObsHolLoc{\alpha}(x, \delta', A')$ for some $x\in\R$, then for all $s, t\in (x-\delta'', x+\delta'')$ we have
$\left|\log f'(t) - \log f'(s) \right| \leq A''|t-s|^{\alpha}$.
\end{thr}
This theorem is proved in Section~\ref{sectSchwLineHolder}.
Here we state an immediate Corollary, which allows to describe very useful compact subsets of $\CZPlus (\R)$.
This corollary is crucial for proving tightness in Theorem~\ref{thrLocalLimit}.

\begin{crl}\label{crlCpmctFromObs}
A set $\cmpct \subset \CZPlus (\R)$ is precompact if for any $x\in \R$ there exist $A_x, \delta_x>0$ such that $\cmpct \subset \ObsHolLoc{\alpha}(x, \delta_x, A_x)$.
\end{crl}
\begin{proof}
Since $\CZPlus(\R)$ is a metric space (with metric given by \eqref{eqDefMetricC1}), it is sufficient to prove sequential compactness.

Let $f_n\in \cmpct$ be some sequence.
We need to show that there exists a converging subsequence.

Fix an interval $[-N, N]$.
Using Theorem~\ref{thrLocObsHolderLine}, the fact that for any $f\in \CZPlus(\R)$ we have $f(0)=0$ and $f'(0)=1$, and Arzel\`a--Ascoli theorem applied to $\log f'$, we can deduce that there exists a subsequence $\{f_{m(N, n)}\}_{n=1}^{\infty}$ of $\{f_{n}\}_{n=1}^{\infty}$ which converges in $C^1\big([-N, N]\big)$.

By a diagonalisation argument we obtain a subsequence $\{f_{m(n, n)}\}_{n=1}^{\infty}$ of $\{f_{n}\}_{n=1}^{\infty}$ which converges on every interval $[-n, n]$ with $n>0$.
By the definition of the metrics \eqref{eqDefMetricC1} on $\CZPlus(\R)$, we deduce that this subsequence converges in $\CZPlus(\R)$.
\end{proof}

\subsection{Ideas of the proofs}
\label{sectProofStrat}

Below we outline some of the key ideas behind the proofs of Theorems~\ref{thrLocalLimit}, \ref{thrGlobalConvJumpBound}, \ref{thrGlobalConvJumpPossible}.

\subsubsection{Ideas of the Theorem~\ref{thrLocalLimit} proof}
In order to prove the Theorem, we replace the map $\Emb{\beta}$, defined in~\eqref{eqDefFMap}, with another map, which is more suitable in this context.
Let us explain what this other map is, and what are its benefits. 

Instead of \eqref{eqDefFMap}, one can consider the map 
\begin{equation}\label{eqDefFaltern}
\phi \mapsto 
\tfrac{\beta}{\pi} \tan\left(\tfrac{\pi}{\beta} \,
\Emb{\beta}_\phi (t)
\right) 
=
\tfrac{\beta}{\pi} \tan\left(\pi \,
\GProj_\phi\big(\tfrac{t}{\beta}\big)
\right)
,
\end{equation}
where the right-hand side is viewed as a function of $t\in \R$.
Informally, for large $\beta$, this map becomes asymptotically equivalent to $\Emb{\beta}$, since $\forall y\in \R: \, \lim_{\beta\to\infty} \frac{\beta}{\pi} \tan (\frac{\pi}{\beta} \, y ) =y$.
Then, informally, Theorem~\ref{thrLocalLimit} says that push-forwards of $\mathcal{Z}^{-1}_{\sigma^2}\cdot\d\FMeas{\sigma^2}$ under this map converge to $\d\FMInf{\varkappa^2}$ whenever $\sigma, \beta\to\infty$ with $\sigma^2/\beta \to \varkappa^2$.

The advantage of the map~\eqref{eqDefFaltern}, compared to $\Emb{\beta}$, is that it behaves much nicer under $\SL(2, \R)$ action on $\phi$.
Indeed, the $\SL(2, \R)$ action on $\phi$ coincides with the action by post-compositions of the group of fractional linear transformations  on $\tan(\pi \phi)$.
This means that changing the gauge (i.e. replacing $\GProj_\phi$ with another representative of the conjugacy class) is the same as applying a fractional linear transformation to the right-hand side of~\eqref{eqDefFaltern}.
Note that this is precisely the natural action of $\SL(2, \R)$ in the context of the Schwarzian Measure on the real line, see Remark~\ref{rmrkSLLineAction}.

Moreover, natural observables for the Schwarzian Field Theory on the circle, given by
\begin{equation}
\obs{\phi}{s}{t} = \frac{\pi\sqrt{\phi'(t)\phi'(s)}}{\sin\Big(\pi\big|\phi(t)-\phi(s)\big|\Big)},
\end{equation}
can be nicely expressed in terms of~\eqref{eqDefFaltern},
\begin{equation}
\obs{\phi}{s}{t}
=
\beta\obs{\tfrac{\beta}{\pi} \tan\big( \tfrac{\pi}{\beta} \, \Emb{\beta}_{\phi}\big)}{s \beta}{t \beta}_0.
\end{equation}
This fits well with the fact that observables $\obs{f}{\cdot}{\cdot}_0$ are precisely the natural observables for the Schwarzian Theory on the real line. 
Importantly, we can explicitly calculate correlation functions of $\obs{f}{\cdot}{\cdot}_0$ for the Schwarzian Theory on the real line, prove the corresponding uniqueness theorem, and, using them, formulate an analogue of the H\"{o}lder Property (see Section~\ref{sectSchwLineIntro} for the details). 

\medskip

However, working with the map \eqref{eqDefFaltern} also introduces some technical difficulties.
Crucially, the right-hand side of \eqref{eqDefFaltern} is discontinuous (and not even c\`{a}dl\`{a}g).
Thus, much greater care is needed when choosing the space of convergence for measures in Theorem~\ref{thrLocalLimit}.

One way to address this problem is to regularise the right-hand side of~\eqref{eqDefFaltern} to make it a $C^1(\R)$ function.
In Section~\ref{sectThrLocalProof}, where we prove Theorem~\ref{thrLocalLimit}, we introduce the map $\phi\mapsto \EmbTan{\beta}_{\phi}$, where function $\EmbTan{\beta}_{\phi}(t)$ coincides with the right-hand side of~\eqref{eqDefFaltern} for $t \in [-\beta/3, \beta/3]$ and is continued as a linear function for $t \notin [-\beta/3, \beta/3]$.
Thus, $\EmbTan{\beta}_{\phi}$ benefits from the properties described above on $[-\beta/3, \beta/3]$, but also has the advantage of being a $C^1(\R)$ function.

\paragraph{Proof sketch.}
The main part of the Theorem~\ref{thrLocalLimit} proof is the demonstration of the fact that for $\EmbTan{\beta}$, described above and spelled out again in \eqref{eqDefFTan}, we have that normalised push-forwards $\PartF_{\sigma^2}^{-1}\cdot \d\EmbTan{\beta}_{\sharp}\FMeas{\sigma^2}$ converge weakly to $\d\FMInf{\varkappa^2}$.

In order to do this, we first show that Schwarzian Measure on the line $\d\FMInf{\varkappa^2}$ is characterised by its correlation functions (see Theorem~\ref{thrUniqLineSchw}) and that correlation functions of $\PartF_{\sigma^2}^{-1}\cdot \d\EmbTan{\beta}_{\sharp}\FMeas{\sigma^2}$ converge to those of $\d\FMInf{\varkappa^2}$ (see Proposition~\ref{prpCorrConv}).

Secondly, we show that measures $\PartF_{\sigma^2}^{-1}\cdot \d\EmbTan{\beta}_{\sharp}\FMeas{\sigma^2}$ are tight.
In Theorem~\ref{thrLocObsHolderLine} we show that local H\"{o}lder Property for functions $f\in C^1(\R)$ can be characterised in terms of observables $\obs{f}{\cdot}{\cdot}_0$.
Therefore, we can find large compact sets $\cmpct\subset \CZPlus(\R)$ that are characterised in terms of observables $\obs{f}{\cdot}{\cdot}_0$, see Corollary~\ref{crlCpmctFromObs}.
After that we demonstrate that conditions of this characterisation are satisfied with probability close to $1$, see Proposition~\ref{prpTightnessLoc}.

Thus, we conclude that, indeed, probability measures $\PartF_{\sigma^2}^{-1}\cdot \d\EmbTan{\beta}_{\sharp}\FMeas{\sigma^2}$ converge weakly to $\d\FMInf{\varkappa^2}$.

Finally, in Section~\ref{sectThrLocalProofFinal} we show that limits of $\PartF_{\sigma^2}^{-1}\cdot \d\EmbTan{\beta}_{\sharp}\FMeas{\sigma^2}$ and $\PartF_{\sigma^2}^{-1}\cdot \d\Emb{\beta}_{\sharp}\FMeas{\sigma^2}$ coincide.

\subsubsection{Ideas of the Theorem~\ref{thrGlobalConvJumpBound} proof}
The first observation is that various changes (including the presence or absence of jumps) of $\phi\in \Diff^1(\T)$ can be encoded via cross-ratios
\begin{equation}
\cross{\phi}{t_1}{t_2}{t_3}{t_4}
=
\frac{\obs{\phi}{t_1}{t_3} \, \obs{\phi}{t_2}{t_4}}{\obs{\phi}{t_1}{t_2} \, \obs{\phi}{t_3}{t_4}}
=
\frac{\sin\Big(\pi\big[\phi(t_2)-\phi(t_1)\big]\Big) \, \sin\Big(\pi\big[\phi(t_4)-\phi(t_3)\big]\Big)}
{\sin\Big(\pi\big[\phi(t_3)-\phi(t_1)\big]\Big) \, \sin\Big(\pi\big[\phi(t_4)-\phi(t_2)\big]\Big)}.
\end{equation} 
Indeed, in Proposition~\ref{prpGlobalProofObsJump} we show that we always have that $\cross{\phi}{t_1}{t_2}{t_3}{t_4} \in (0, 1)$, and that if $\phi$ increases on all $4$ intervals $\{(t_j, t_{j+1})\}_{j=1}^4$ then $\cross{\phi}{t_1}{t_2}{t_3}{t_4}$ is bounded away from $0$ and $1$.
Thus, in order to show that for large $\sigma$ probability measures $\PartF_{\sigma^2}^{-1} \cdot \d\FMeas{\sigma^2}$ concentrate on jump processes with no more than $3$ jumps, it is sufficient to show that 
\begin{equation}\label{eqIdeasProofCrossRatio1}
\lim_{\sigma\to \infty} 
\left[
\PartF_{\sigma^2}^{-1} \cdot 
\int_{\Diff^1(\T)/\SL(2, \R)}
\cross{\phi}{t_1}{t_2}{t_3}{t_4}
\left(1- \cross{\phi}{t_1}{t_2}{t_3}{t_4}\right) \d\FMeas{\sigma^2}(\phi) 
\right]
= 0
\end{equation}

In Lemma~\ref{lmmGlobalUniqFuncEqual} we show that we can relate expectations of functions of observables $\obs{\phi}{\cdot}{\cdot}$ with respect to $\d\FMeas{\sigma^2}(\phi)$ to expectations of functions of $\obs{\B_{\xi}}{\cdot}{\cdot}_0$ with respect to Brownian Bridges (see Section~\ref{sect_meas_construct} for the Brownian Bridge definition).
Thus, we reduce \eqref{eqIdeasProofCrossRatio1} to the calculation of cross-ratio $\sigma\to \infty$ limits for the Brownian Bridges measures.
The corresponding Brownian Bridge calculation is carried out in Lemma~\ref{lmmCrossBBConv}.
Combining Lemma~\ref{lmmGlobalUniqFuncEqual} and Lemma~\ref{lmmCrossBBConv} we obtain
\eqref{eqIdeasProofCrossRatio1}, finishing the proof.

\subsubsection{Ideas of the Theorem~\ref{thrGlobalConvJumpPossible} proof}

First of all, we prove the Theorem for $k=3$.
We fix the gauge $\phi(0)=0$, $\phi(1/3)=1/3$, $\phi(2/3) = 2/3$.
This already ensures that there are $3$ jumps.
It remains to show that these jumps are not near the points $0, 1/3, 2/3$.
We do this by demonstrating that cross-ratios $\cross{\phi}{t_1}{t_2}{t_3}{t_4}$ are close to $0$ whenever $t_3$ and $t_4$ are close, and applying this to $0, 1/3, 2/3$ and points near them.
We prove that $\cross{\phi}{t_1}{t_2}{t_3}{t_4}$ are small by combining Lemma~\ref{lmmGlobalUniqFuncEqual} and Lemma~\ref{lmmCrossBBConv}.

The result for $k=1, 2$ follows from the $k=3$ case by changing the gauge so that $1$ or $2$ of the jumps disappear.

\subsection{Organisation of the paper and notations}
\label{sectNotation}
We recall all the necessary mathematical results concerning Schwarzian Field Theory on the circle in Section~\ref{sect_meas_construct}.

In Section~\ref{sectSchwLineProofs} we prove the main results about the Schwarzian Measure on the real line, which we stated in Section~\ref{sectSchwLineIntro}.

We prove Theorem~\ref{thrLocalLimit} in Section~\ref{sectThrLocalProof}.
In Section~\ref{sectThrLocalProofCharacter} we describe the limit of $\PartF_{\sigma^2}^{-1}\cdot \d\EmbTan{\beta}_{\sharp}\FMeas{\sigma^2}$ by its correlation functions.
In Section~\ref{sectThrLocalProofTight} we show that measures $\PartF_{\sigma^2}^{-1}\cdot \d\EmbTan{\beta}_{\sharp}\FMeas{\sigma^2}$ are tight.
In Section~\ref{sectThrLocalProofFinal} we show that limits of $\PartF_{\sigma^2}^{-1}\cdot \d\EmbTan{\beta}_{\sharp}\FMeas{\sigma^2}$ and $\PartF_{\sigma^2}^{-1}\cdot \d\Emb{\beta}_{\sharp}\FMeas{\sigma^2}$ coincide, finishing the proof.

We prove Theorem~\ref{thrGlobalConvJumpBound} and Theorem~\ref{thrGlobalConvJumpPossible} in Section~\ref{sectThrGlobProof}.

\medskip

Throughout the paper we use the following notations:

\begin{enumerate}

\item The unit circle is denoted by $\T=[0,1]/\{0\sim 1\}$, the nonnegative real numbers are denoted by $\R_+ = [0, \infty)$, and for the open disk in the complex plane of radius $r$ we use $\D_r = \left\{z\in \Compl: |z|<r\right\}$.
Moreover, for $s, t\in \T$, use $(s, t)$ to denote the interval going from $s$ to $t$ in positive direction.
We write $t-s$ for the length of the interval $(s, t)$.
In particular, $t-s\in [0,1)$.

\item We use $\Diff^k(\T)$ for the set of orientation-preserving $C^k$-diffeomorphisms of $\T$, that is, $\phi\in \Diff^k(\T)$
can be identified with a $k$-times continuously differentiable function $\phi\colon\R\to \R$ satisfying
$\phi(\tau+1)=\phi(\tau)+1$ and $\phi'(\tau)>0$ for all $\tau \in \R$. Note that $\Diff^k(\T)$ is not a linear space.
The topology on $\Diff^k(\T)$ is inherited from the natural topology on $C^k(\T)$.
It turns $\Diff^{k}(\T)$ into a Polish (separable completely metrisable) space as well as a topological group.

\item We use $\SL(2, \R)$ to denote the group of conformal isomorphisms of the unit disk restricted to the boundary, which is identified with the unit circle $\T$.
Throughout the paper we consider the action of $\SL(2, \R)$ on $\Diff^1(\T)$ given by post-compositions.

\item We will often abuse the notation and for $\phi\in \Diff^1(\T)$ denote its conjugacy class in $\Diff^1(\T)/\SL(2,\R)$ by the same symbol $\phi$.
When this convention may lead to ambiguity, we will instead write $[\phi]\in \Diff^1(\T)/\SL(2,\R)$ to denote the conjugacy class explicitly.

\item In some places we will encounter expressions of the form $f\big(\arccosh[z]\big)$, for various even analytic functions $f$.
Even though $\arccosh[z]$ is not analytic at $z=0$, the composition $f\big(\arccosh[z]\big)$ still defines an analytic function around $z=0$.
More precisely, we identify $f\big(\arccosh[z]\big)$ with $\widetilde{f}\big(\arccosh^2[z]\big)$, where $\widetilde{f}$ is an analytic function such that $\widetilde{f}(\omega) = f(\sqrt{\omega})$, and  $\arccosh^2[z]$ is the analytic function described in the statement below (see \citep{LosevCorr} for details).
\begin{stm}[\citep{LosevCorr}]\label{stmArccoshDef}
Function $\arccosh^2(z)$ can be analytically continued from $z\in [1, \infty)$ to $z \in \Dom := \left\{z\in\Compl \, \Big| \, \Real z\geq -1 \right\}$. 
Moreover, in this continuation $\arccosh^2(z) = -\arccos^2(z)$ for $z \in (-1, 1)$.
\end{stm}
\end{enumerate}

\section{Schwarzian Measure on the circle}\label{sect_meas_construct}

\subsection{Definitions and partition function}
In this subsection we recall the rigorous construction and main properties of the measure corresponding to Schwarzian Field Theory from \citep*{BLW}, which are based on the plan from \citep{BelokurovShavgulidzeExactSolutionSchwarz, BelokurovShavgulidzeCorrelationFunctionsSchwarz}. 
We refer the reader to that paper for proofs and further details.
\medskip

The construction of the Schwarzian measure is based on the appropriate reparametrisation of an unnormalised version of the Brownian bridge measure.
This is a finite measure on $\Cfree[0,T] = \left\{f\in C[0,T]\, | \, f(0)=0\right\}$ formally corresponding to
\begin{equation} \label{e:BB-formaldensity}
\d\WS{\sigma^2}{a}{T}(\xi) = \exp\left\{-\frac{1}{2\sigma^2}\int_{0}^{T}\xi'^{\, 2} (t)\d t\right\} \delta\big(\xi(0)\big) \delta\big(\xi(T)-a\big)\prod_{\tau \in (0,T)}\d\xi(\tau).
\end{equation}

\begin{defn} \label{defn:BB}
  The unnormalised Brownian bridge measure with variance $\sigma^2 > 0$ is a finite Borel measure $\d\WS{\sigma^2}{a}{T}$ on $\Cfree[0, T]$ such that
  \begin{equation}\label{eq:25}
    \sqrt{2\pi T}\sigma \, \exp\left\{\frac{a^2}{2 T\sigma^2}\right\}\d\WS{\sigma^2}{a}{T}(\xi)
  \end{equation}
  is the distribution of a Brownian bridge $\big(\xi(t)\big)_{t\in[0,T]}$ with variance $\sigma^2$ and $\xi(0) = 0$, $\xi(T) = a$.
\end{defn}

In order to define the Schwarzian measure $\FMeas{\sigma^2}$,
we first need to define a finite measure $\mu_{\sigma^2}$ on $\Diff^1(\T)$ which is similar to what is known as the Malliavin--Shavgulidze measure;
see \cite[Section~11.5]{BogachevMalliavin}. %
Formally, this measure corresponds to
\begin{equation} \label{e:MeasureMuFormal}
\d\mu_{\sigma^2}(\phi) 
= \exp\left\{-\frac{1}{2\sigma^2}\int_{0}^{1}\left(\frac{\phi''(\tau)}{\phi'(\tau)}\right)^2\d\tau\right\} \prod_{\tau \in [0,1)}\frac{\d\phi(\tau)}{\phi'(\tau)}.
\end{equation}

We can make sense of this measure by defining it as a push-forward of an unnormalised Brownian bridge on $[0,1]$ with respect to a suitable change of variables.
We \emph{define} $\mu_{\sigma^2}$ by
\begin{equation}\label{defMeasureMu}
\d\mu_{\sigma^2}(\phi) \coloneqq 
\d \WS{\sigma^2}{0}{1}(\xi)
\otimes \d\Theta, \qquad \text{with } \phi(t) = \Theta + \A_{\xi}(t)\enspace (\mathrm{mod}\, 1), \text{ for } \Theta\in [0, 1),
\end{equation}
where $\d\Theta$ is the Lebesgue measure on $[0,1)$ and
\begin{equation} \label{defP}
\A(\xi)(t) \coloneqq \A_{\xi}(t) 
\coloneqq \frac{\int_{0}^t e^{\xi(\tau)}\d\tau}{\int_{0}^1 e^{\xi(\tau)}\d\tau}.
\end{equation}
The variable $\Theta$ corresponds to the value of $\phi(0)$.
Note that the map $\xi \mapsto \A(\xi)$ is a bijection between $\Cfree[0, 1]$
and $\Diff^1 [0,1]$ with inverse map
\begin{align} 
\A^{-1}: \Diff^1[0,1] &\to \Cfree[0, 1]\\
\varphi &\mapsto \log \varphi'(\cdot) - \log \varphi'(0).
\end{align}
Note that, with our choice of $C^1$ topology on $\Diff^1[0,1]$ and the usual supremum norm topology on $\Cfree[0, 1]$, we get that both $\A$ and $\A^{-1}$ are continuous.

In view of \eqref{eq:1} and \eqref{e:MeasureMuFormal},
the unquotiented Schwarzian measure is constructed as
\begin{equation} \label{defMeasureMeas}
\d\FMeasSL{\sigma^2}(\phi) = \exp\left\{ \frac{2\pi^2}{\sigma^2}\int_{0}^{1} \phi'^{\, 2}(\tau)\d\tau\right\}\d \mu_{\sigma^2}(\phi).
\end{equation}
Since $\mu_{\sigma^2}$ is supported on $\Diff^1(\T)$, this defines a Borel measure on $\Diff^1(\T)$. 
This is the unique (up to a multiplicative constant) measure satisfying a natural change of variables formula.
In particular, it is invariant under the action of $\SL(2,\R)$.
\begin{prp}\label{lemmaMeasureSLInv}
The measure $\FMeasSL{\sigma^2}$ is invariant under post-composition by elements of $\SL(2, \R)$. 
In other words, for any $\psi \in \SL(2, \R)$ and Borel $A\subset \Diff^1(\T)$ we have
\begin{equation}
  \FMeasSL{\sigma^2}\big(\psi\circ A\big) = \FMeasSL{\sigma^2}\big(A\big),
\end{equation}
where  $\psi\circ A \coloneqq \left\{\psi\circ \phi\, \big | \, \phi \in A \right\}$.
\end{prp}

Since $\mu_{\sigma^2}$ is supported on $\Diff^1(\T)$, this defines a Borel measure on $\Diff^1(\T)$, 
which turns out to be the unique measure satisfying the natural change of variables formula.

\begin{thr_old}[Theorem~1, \citep*{BLW}]
The measure $\FMeasSL{\sigma^2}$ is the unique (up to a multiplicative constant) $\SL(2,\R)$-invariant Borel measure supported on $\Diff^1(\T)$
  that satisfies the expected change of variables formula
  \begin{equation}  \label{eqDiffeoMeasureChange}
    \frac{\d \psi^{\ast}\!\FMeasSL{\sigma^2}(\phi)}{\d\FMeasSL{\sigma^2}(\phi)} =
    \frac{\d \FMeasSL{\sigma^2}(\psi\circ \phi)}{\d\FMeasSL{\sigma^2}(\phi)} = 
    \exp\left\{
      \frac{1}{\sigma^2}\int_{\T}\Big[ \Schw(\tan(\pi\psi),\phi(\tau))-2\pi^2 \Big]\, \phi'(\tau)^2 \d \tau \right\}
    ,
  \end{equation}
    for  any $\psi\in \Diff^3(\T)$, and has a quotient $\FMeas{\sigma^2} = \FMeasSL{\sigma^{2}}/\SL(2,\R)$ that is a finite Borel measure on $\Diff^1(\T)/\SL(2,\R)$.
\end{thr_old}

\begin{defn}\label{defn_schwarzian}
  The Schwarzian measure is given by $\FMeas{\sigma^2}$.
\end{defn}
This measure is finite, and moreover, its total mass can be computed explicitly.

\begin{prp}\label{prpMainPartFunct}
The partition function (i.e., total mass) of $\FMeas{\sigma^2}$ is given by
\begin{equation}
\PartF(\sigma^2) =  
    \left(\frac{2\pi}{\sigma^2}\right)^{3/2} \exp\left(\frac{2\pi^2}{\sigma^2}\right)
    =\int_0^{\infty} e^{-{\sigma^2k^2}/{2}} \sinh(2\pi k) \, 2 k \d k.
\end{equation}
\end{prp}

\subsection{Observables}
Here we recall the main properties of cross-ratio observables
\begin{equation}\label{eqDefObs_intro}
\obs{\phi}{s}{t} = \frac{\pi\sqrt{\phi'(t)\phi'(s)}}{\sin\Big(\pi\big|\phi(t)-\phi(s)\big|\Big)},
 \qquad s,t\in \T,
\end{equation}  
obtained in \citep{LosevCorr}.

First of all, these expressions indeed define observables on the space $\Diff^1(\T)/\SL(2, \R)$, since they depend only on the conjugacy class under $\SL(2, \R)$ action.
\begin{prp}\label{prpObsSLInvar}
Observables $\obs{\phi}{s}{t}$ are invariant under post-compositions by M\"{o}bius transformations. 
In other words, if $\psi\in \SL(2, \R)$, then
\begin{equation}
\obs{\psi\circ\phi}{s}{t} = \obs{\phi}{s}{t}.
\end{equation}
In particular, they induce well-defined observables on $\Diff^1(\T)/\SL(2, \R)$, which we, slightly abusing the notation, also denote by $ \obs{\phi}{\cdot}{\cdot}$.
\end{prp}
Importantly, it is possible to compute correlation functions of cross-ratios \citep{LosevCorr}.
We use the same diagrammatic language, as in Section~\ref{sectSchwLineIntro}. 
The only difference is that now we replace the line with the circle and semi-circular chords with linear chords (see~\citep{LosevCorr} for more details).

\begin{thr_old} \label{thrMainCorrelations_corr}
For $N\geq 0$ let $\big\{\obs{\phi}{s_{j}}{t_j}\big\}_{j=1}^{N}$ be a set of non-interlaced observables, and $\{l_j\}_{j=1}^N$ be a set of positive integers. 
Let also $\{k_m\}_{m=1}^{N+1}$ and $\{\tau_m\}_{m=1}^{N+1}$ be as above.
Then 
\begin{multline}
\int\limits_{{\Diff^1(\T)/\SL(2, \R)}} \prod_{j=1}^N \obs{\phi}{s_j}{t_j}^{l_j}\d\FMeas{\sigma^2}(\phi)
=
\int_{\R_+^{N+1}} 
\prod_{j=1}^N \frac{\Gamma\left(\dfrac{l_j}{2}\pm i w_1(j) \pm i w_2(j)\right)}{2\pi^2 \, \Gamma(l_j)}
\cdot\left(\frac{\sigma^2}{2}\right)^{l_j}\\
\times
\prod_{m=1}^{N+1}\exp\left(-\frac{\tau_m \sigma^2}{2}\cdot k_m^2\right)  \sinh(2\pi k_m) \, 2 k_m \d k_m,
\end{multline}
where the integral on the right-hand side converges absolutely.
\end{thr_old}

Using this, it is possible to show that cross-ratios have finite exponential moments.
\begin{prp}\label{prpExpMomentCirc_corr}
For any $\sigma>0$ and any $s\neq t \in \T$, 
\begin{equation}
\int  \exp\left\{ \frac{8}{\sigma^2}\, \obs{\phi}{s}{t} \right\} \d\FMeas{\sigma^2}(\phi) < \infty.
\end{equation}
\end{prp}
In this work we give a sharper estimate for the exponential moments in the case when $s$ and $t$ are close; see Proposition~\ref{prpBoundExpMomentsLine}.

\subsection{Schwarzian Measure on the interval}
\label{sectSchwInterval}

In \citep{LosevCorr}, when we were calculating the correlation functions for the Schwarzian Theory on the circle, we also derived several results about the correlation functions for the push-forwards of the Brownian Bridge measures $\d \WS{\sigma^2}{a}{T}$ under $\B$, given in \eqref{eqDefQ}.
We call these push-forwards \emph{Schwarzian Measure on the interval} $[0, T]$.
Below we recall results which turn out to be useful in the present work as well.

Firstly, we can calculate correlation functions of non-interlaced observables.
The diagrammatic language for the Proposition below is the same as in Section~\ref{sectSchwLineIntro}, with the only exception that now $\tau_0$ is finite.

\begin{prp}\label{prpIntervalExpObsFormula_corr}
Fix $T>0$, $a\in \R$, and $N\geq 0$.
Let $\Big\{\obs{f}{s_{j}}{t_j}_0\Big\}_{j=1}^{N}$ be a set of non-interlaced observables on $[0, T]$, 
 and $\left\{l_j\right\}_{j=1}^N$ be positive integers. 
Then,
\begin{multline} \label{eqPrpFinalResultInterval_corr}
\int 
\prod_{j=1}^N \obs{\B_{\xi}}{s_j}{t_j}^{l_j}_0
\d \WS{\sigma^2}{a}{T}(\xi)
=
\int_{\R_+^{N+1}} 
\prod_{j=1}^N \frac{\Gamma\Big(\frac{l_j}{2}\pm i w_1(j) \pm i w_2(j)\Big) }{2 \pi^2\, \Gamma(l_j)} \cdot \left(\frac{\sigma^{2}}{2}\right)^{l_j}\\
\times
\exp\left(-\frac{\tau_0\sigma^2}{2}\cdot k_0^2\right)  \frac{\cos(a\, k_0)}{\pi} \d k_0 \cdot
\prod_{m=1}^{N}\exp\left(-\frac{\tau_m \sigma^2}{2}\cdot k_m^2\right)  \sinh(2\pi k_m)\, 2 k_m \d k_m,
\end{multline}
where the right-hand side converges absolutely.
\end{prp}

Secondly, we can bound exponential moments of the cross-ratio observables.

\begin{prp}\label{prpObsExpMomentBound_corr}
Let $T>0$ and $a\in \R$. For any $0\leq s< t\leq T$ we have
\begin{equation}
\int \exp\left\{\frac{8}{\sigma^2}\, \obs{\B_{\xi}}{s}{t}_0 \right\} 
\d \WS{\sigma^2}{a}{T}(\xi) 
\leq
\frac{1}{\sqrt{2\pi T}\sigma}\exp\left(\frac{2000}{(t-s)\sigma^2}-\frac{a^2}{2 T\sigma^2}\right).
\end{equation}
In particular, the expression above is finite.
\end{prp}

\section{Proofs for Schwarzian Measure on the line} 
\label{sectSchwLineProofs}
In this section we prove the main results from Section~\ref{sectSchwLineIntro}.
We prove Theorem~\ref{thrUniqLineSchw} in Section~\ref{sectSchwLineUnique}, and Theorem~\ref{thrLocObsHolderLine} in Section~\ref{sectSchwLineHolder}.

\subsection{Uniqueness theorem}
\label{sectSchwLineUnique}

\begin{lmm}\label{lmmUniqLineExpExten}
In the setting of Theorem~\ref{thrUniqLineSchw}, we have that for any $N>0$ and any continuous bounded function $F\in C(\R^N)$,
\begin{equation}\label{eqUniqFuncEqual}
\int_{\CZPlus(\R)} F\left(\Big\{\obs{f}{s_{j}}{t_j}_0\Big\}_{j=1}^{N}\right) \d\mathcal{P}(f)
=
\int_{\CZPlus(\R)} F\left(\Big\{\obs{f}{s_{j}}{t_j}_0\Big\}_{j=1}^{N}\right) \d\FMInf{\varkappa^2}(f).
\end{equation}
\end{lmm}
\begin{proof}
From the existence of exponential moments (see Corollary~\ref{crlExpMomentLine}), and the fact that \eqref{eqUniqFuncEqual} holds for polynomials, we deduce that \eqref{eqUniqFuncEqual} also holds for imaginary exponentials of linear functions.
In other words, characteristic functions of distributions of stochastic processes $\{\obs{f}{s_{j}}{t_j}_0\Big\}_{j=1}^{N}$ with respect to $\d\mathcal{P}$ and $\d\FMInf{\varkappa^2}$ coincide.
Thus, these distributions coincide, which is equivalent to the statement of this Lemma.
\end{proof}

For the rest of this section it will be useful to note that for any $f\in\CZPlus(\R)$,
\begin{equation}\label{eqCross3P1}
\frac{\obs{f}{0}{s}_0 \obs{f}{0}{t}_0}{\obs{f}{s}{t}_0} 
= \frac{f'(0) \big(f(t)-f(s)\big)}
{f(s) f(t)}
= \frac{1}{f(s)} -\frac{1}{f(t)}, \qquad \text{for } 0<s<t.
\end{equation}
Similarly,
\begin{align}\label{eqCross3P2}
\frac{\obs{f}{s}{0}_0 \obs{f}{0}{t}_0}{\obs{f}{s}{t}_0} 
&= \frac{1}{f(t)} -\frac{1}{f(s)}, \qquad \text{for } s<0<t;\\
\label{eqCross3P3}
\frac{\obs{f}{s}{0}_0 \obs{f}{t}{0}_0}{\obs{f}{s}{t}_0} 
&= \frac{1}{f(s)} -\frac{1}{f(t)}, \qquad \text{for } s<t<0.
\end{align}

\begin{lmm}
In the setting of Theorem~\ref{thrUniqLineSchw}, we have that  $\lim_{T\to\infty}f(-T)=-\infty$ and $\lim_{T\to\infty}f(T)=\infty$ almost surely under both $\d\mathcal{P}$ and $\d\FMInf{\varkappa^2}$.
In other words,
\begin{align}
\mathcal{P}\left\{f: \, \lim_{T\to\infty}f(-T) = -\infty, \text{ and }   \lim_{T\to\infty}f(T) =\infty \right\} &=1;\\
\FMInf{\varkappa^2}\left\{f: \, \lim_{T\to\infty}f(-T) = -\infty, \text{ and }   \lim_{T\to\infty}f(T) =\infty\right\} &=1.
\end{align}
\end{lmm}

\begin{proof}
Since $f$ is monotone, it is sufficient to show that for any $\eps>0$
\begin{align}\label{eqUniqFTLimit1}
\lim_{T\to\infty}\mathcal{P}\left\{f:\,\frac{1}{|f(-T)|}+\frac{1}{|f(T)|}>\eps\right\} &= 0;\\
\label{eqUniqFTLimit2}
\lim_{T\to\infty}\FMInf{\varkappa^2}\left\{f:\,\frac{1}{|f(-T)|}+\frac{1}{|f(T)|}>\eps\right\} &= 0.
\end{align}

Using \eqref{eqCross3P2} and Lemma~\ref{lmmUniqLineExpExten},
\begin{multline}
\mathcal{P}\left\{f:\, \frac{1}{|f(-T)|}+\frac{1}{|f(T)|}>\eps\right\}
=
\mathcal{P}\left\{f:\,\frac{\obs{f}{-T}{0}_0 \obs{f}{0}{T}_0}{\obs{f}{-T}{T}_0}>\eps\right\}\\
=
\FMInf{\varkappa^2}\left\{f:\,\frac{\obs{f}{-T}{0}_0 \obs{f}{0}{T}_0}{\obs{f}{-T}{T}_0}>\eps\right\}
=
\FMInf{\varkappa^2}\left\{f:\,\frac{1}{|f(-T)|}+\frac{1}{|f(T)|}>\eps\right\}
\end{multline}
Therefore, it is enough to prove \eqref{eqUniqFTLimit2}.

Recall that by definition \eqref{eqDefLineSchw} of $\d\FMInf{\varkappa^2}$,
\begin{align}
\FMInf{\varkappa^2}\left\{f:\,\frac{1}{|f(-T)|}+\frac{1}{|f(T)|}>\eps\right\}
&=
\BM{\varkappa^2}{\infty}\left\{\xi:\, \left(\int_{-T}^{0} e^{\xi(s)}\d s\right)^{-1} + \left(\int_{0}^{T} e^{\xi(s)}\d s\right)^{-1}>\eps\right\}\\
&\leq
\BM{\varkappa^2}{\infty}\left\{\xi:\, \big|\left\{s\in[-T,0]:\, \xi(s)\geq 0\right\}\big|<2\eps^{-1}\right\}\\
&+\BM{\varkappa^2}{\infty}\left\{\xi:\, \big|\left\{s\in[0, T]:\, \xi(s)\geq 0\right\}\big|<2\eps^{-1}\right\}.
\end{align}
Taking $T\to\infty$ limit proves \eqref{eqUniqFTLimit2}.
\end{proof}

\medskip

\begin{proof}[Proof of Theorem~\ref{thrUniqLineSchw}]
It is sufficient to prove the equality for cylinder sets.
Fix $t_{-N} < t_{-N+1} < \ldots <t_0=0<t_1<\ldots<t_N$. 
We also take a large $T>t_N$, aiming to consider a $T\to \infty$ limit later.

As a shorthand, we index observables $\Big\{\obs{f}{t_{j}}{t_{j+1}}_0\Big\}_{j=-N}^{N-1}\cup\Big\{\obs{f}{0}{t_j}_0\Big\}_{j\notin\{-1, 0, 1\}}$ by an index set $I$ (with $|I| = 4N-2$), and denote them by $\mathcal{O}_I(f) = \big\{\mathcal{O}_i(f) \big\}_{i\in I}$.

We now study the stochastic process of $\mathcal{O}_I(f)$ together with $\Big\{\obs{f}{t_{j}}{T}_0\Big\}_{j\in\{-N, 0, N\}}$. 
Note that this stochastic process consists of a non-interlaced set of observables, so Lemma~\ref{lmmUniqLineExpExten} holds for them.

Using \eqref{eqCross3P1}, \eqref{eqCross3P2}, and \eqref{eqCross3P3} we can express $\left\{\tfrac{1}{f(t_j)}-\tfrac{1}{f(t_{j+1})}\right\}_{j=-N}^{-2}$ and $\left\{\tfrac{1}{f(t_j)}-\tfrac{1}{f(t_{j+1})}\right\}_{j=1}^{N-1}$ as algebraic expressions of $\mathcal{O}_I(f)$.
Thus, it is possible to write $\{f(t_j)\}_{j\neq 0}$ as algebraic expressions of $\mathcal{O}_I(f)$ and $\{f(t_{-N}), f(t_N)\}$.
Since we also have that
\begin{equation}
\obs{f}{0}{t_j}_0 = \frac{\sqrt{f'(t_j)}}{|f(t_j)|},
\end{equation}
we deduce that $\big\{f(t_i), f'(t_i)\big\}_{i\neq 0}$ can be written as algebraic expressions of $\mathcal{O}_I(f)$ and $\{f(t_{-N}), f(t_N)\}$.

Therefore, we conclude that any continuous bounded function with compact support $G:\R^{4N}\to\R$ of $\big\{f(t_i), f'(t_i)\big\}_{i\neq 0}$ can be expressed as a bounded continuous function $H:\R^{4N}\to \R$ of $\mathcal{O}_I(f)$, $f(t_{-N})$, and $f(t_{N})$.

We denote (recall \eqref{eqCross3P1} and \eqref{eqCross3P2})
\begin{align}
l_1(f, T) &\coloneqq - \frac{\obs{f}{t_{-N}}{T}_0}{\obs{f}{t_{-N}}{0}_0 \obs{f}{0}{T}_0}
= - \left(\frac{1}{f(T)} - \frac{1}{f(t_{-N})}\right)^{-1} \\
l_2(f, T) &\coloneqq \frac{\obs{f}{t_N}{T}_0}{\obs{f}{0}{t_N}_0 \obs{f}{0}{T}_0}
= \left(\frac{1}{f(t_N)} -\frac{1}{f(T)}\right)^{-1}.
\end{align}

Since 
\begin{align}
\lim_{T\to \infty} l_1(f, T) &= f(t_{-N})\\
\lim_{T\to \infty} l_2(f, T) &= f(t_{N}),
\end{align}
by Dominated Convergence Theorem we get
\begin{multline}
\lim_{T\to\infty}
\int H\Big(\mathcal{O}_I(f)\cup\{l_1(f, T), l_2(f, T)\}\Big)\d\mathcal{P}(f)
=
\int H\Big(\mathcal{O}_I(f)\cup\{f(t_{-N}), f(t_{N})\}\Big)\d\mathcal{P}(f)\\
=
\int G\left(\big\{f(t_i), f'(t_i)\big\}_{i\neq 0}\right)\d\mathcal{P}(f)
\end{multline}
and
\begin{multline}
\lim_{T\to\infty}
\int H\Big(\mathcal{O}_I(f)\cup\{l_1(f, T), l_2(f, T)\}\Big)\d\FMInf{\varkappa^2}(f)
=
\int H\Big(\mathcal{O}_I(f)\cup\{f(t_{-N}), f(t_{N})\}\Big)\d\FMInf{\varkappa^2}(f)\\
=
\int G\left(\big\{f(t_i), f'(t_i)\big\}_{i\neq 0}\right)\d\FMInf{\varkappa^2}(f).
\end{multline}
Moreover, by Lemma~\ref{lmmUniqLineExpExten},
\begin{equation}
\int H\Big(\mathcal{O}_I(f)\cup\{l_1(f, T), l_2(f, T)\}\Big)\d\mathcal{P}(f)
=
\int H\Big(\mathcal{O}_I(f)\cup\{l_1(f, T), l_2(f, T)\}\Big)\d\FMInf{\varkappa^2}(f),
\end{equation}
which finishes the proof.
\end{proof}

\subsection{H\"{o}lder characterisation}
\label{sectSchwLineHolder}

\begin{proof}[Proof of Theorem~\ref{thrLocObsHolderLine}]
$ $

\textbf{Direct statement.}
It is sufficient to prove the result only for $t>s$.

In this proof $C_1, C_2, \ldots$ denote constants that depend only on $A$, $\delta$, and $\alpha$ (in particular, they do not depend on $f$).
We also assume $\eps = \eps (A, \delta, \alpha)>0$ is sufficiently small for the argument below to work.

Exponentiating 
\begin{equation}
\left|\log f'(s+\tau) - \log f'(s) \right| \leq A \tau^{\alpha}
\end{equation}
we get
\begin{equation}\label{eqThrHolderProof2}
\left|\frac{f'(s+\tau)}{f'(s)} -1 \right| \leq C_1 \tau^{\alpha}.
\end{equation}
Integrating over $\tau$ from $0$ to $t-s$, we obtain for all $t\in (s, s+\eps)$,
\begin{equation}\label{eqThrHolderProof3}
\left|\frac{f(t)-f(s)}{f'(s)} - |t-s| \right| \leq C_2 |t-s|^{1+\alpha}.
\end{equation}
Using \eqref{eqThrHolderProof2} for $\tau = t-s$, we also have
\begin{equation}\label{eqThrHolderProof4}
\left|\frac{f'(t)}{f'(s)} -1 \right| \leq C_1 |t-s|^{\alpha}
\end{equation} 
Combining \eqref{eqThrHolderProof3} and \eqref{eqThrHolderProof4},
\begin{equation}
\left|\frac{f(t)-f(s)}{\sqrt{f'(t)f'(s)}} - |t-s| \right| \leq C_3 |t-s|^{1+\alpha}.
\end{equation}
Taking the inverse proves the desired result.

\medskip

\textbf{Converse statement.}
It is sufficient to prove the result only for $t>s$.

In this proof $C_1, C_2, \ldots$ denote constants that depend only on $A'$, $\delta'$, and $\alpha$ (in particular, they do not depend on $f$).
We also assume $\eps = \eps (A', \delta', \alpha)>0$ is sufficiently small for the argument below to work.

Using 
\begin{equation}
\frac{\obs{f}{s-\eps}{s}_0 \obs{f}{s}{s+\eps}_0}{\obs{f}{s-\eps}{s+\eps}_0}
=
f'(s)\left( \frac{1}{f(s) - f\left(s-\eps\right)} + \frac{1}{f\left(s+\eps\right) - f(s)}  \right).
\end{equation}
we get that for $f\in \ObsHolLoc{\alpha}(x, \delta', A')$,
\begin{equation}
f'(s)\left( \frac{1}{f(s) - f\left(s-\eps\right)} + \frac{1}{f\left(s+\eps\right) - f(s)}  \right)
\leq 
2\eps^{-1} + C_1 \eps^{\alpha-1}
\end{equation}
Thus,
\begin{equation}\label{eqHolLineProof1}
\frac{f'(s)}{f\left(s+\eps\right) - f(s)} \leq C_2\eps^{-1} .
\end{equation}

Since $f\in \ObsHolLoc{\alpha}(x, \delta', A')$, by considering $\obs{f}{s}{s+\tau}_0^2$ we get
\begin{equation}
\left|\frac{f'(s) f'(s+\tau)}{|f(s+\tau) - f(s)|^2}- \frac{1}{\tau^2}\right| \leq C_3\tau^{\alpha-2}.
\end{equation}
Integrating in $\tau$ from some $\eta \in (0, \eps^{-1/(\alpha-1)})$ to $\eps$ we obtain
\begin{equation}
\left|\frac{f'(s)}{f(s+\eta) - f(s)} - \frac{f'(s)}{f(s+\eps) - f(s)} - \frac{1}{\eta}\right| \leq C_4\eta^{\alpha-1}.
\end{equation}
Using \eqref{eqHolLineProof1} we obtain
\begin{equation}
\left|\frac{f'(s)}{f(s+\eta) - f(s)} - \frac{1}{\eta}\right| \leq C_5\eta^{\alpha-1}.
\end{equation}
Dividing the inequality above by $\obs{f}{s}{s+\eta}_0$ and using that for $f\in \ObsHolLoc{\alpha}(x, \delta', A')$
\begin{equation}
\left|\obs{f}{s}{s+\eta}_0 - \frac{1}{\eta}\right| \leq A' \eta^{\alpha-1},
\end{equation}
we get
\begin{equation}
\left|\sqrt{\frac{f'(s)}{f'(s+\eta)}} - 1\right| \leq C_6\eta^{\alpha}.
\end{equation}
The result now follows by taking the logarithm.
\end{proof}

\section{Local structure proofs}
\label{sectThrLocalProof}

The main goal of this section is to prove Theorem~\ref{thrLocalLimit}.
In Section~\ref{sectThrLocalProofCharacter} we describe the limit of $\PartF_{\sigma^2}^{-1}\cdot \d\EmbTan{\beta}_{\sharp}\FMeas{\sigma^2}$ by its correlation functions.
In Section~\ref{sectThrLocalProofTight} we demonstrate that measures $\PartF_{\sigma^2}^{-1}\cdot \d\EmbTan{\beta}_{\sharp}\FMeas{\sigma^2}$ are tight.
In Section~\ref{sectThrLocalProofFinal} we finish the proof by showing that limits of $\PartF_{\sigma^2}^{-1}\cdot \d\EmbTan{\beta}_{\sharp}\FMeas{\sigma^2}$ and $\PartF_{\sigma^2}^{-1}\cdot \d\Emb{\beta}_{\sharp}\FMeas{\sigma^2}$ coincide.

As explained in Section~\ref{sectProofStrat}, we start by replacing $\Emb{\beta}$ with $\EmbTan{\beta}$, which fits better with the $\SL(2, \R)$ structure of the whole setting.
We use an embedding $\EmbTan{\beta}:\, \Diff^1(\T)/\SL(2, \R) \to \CZPlus(\R) $ given by
\begin{equation}\label{eqDefFTan}
\phi \mapsto \EmbTan{\beta}(\phi)(t) \coloneqq \EmbTan{\beta}_\phi (t)
\coloneqq 
\begin{cases}
\tfrac{\beta}{\pi} \tan\left(\tfrac{\pi}{\beta} \,
\Emb{\beta}_\phi (t)
\right) 
=
\tfrac{\beta}{\pi} \tan\left(\pi \,
\GProj_\phi\big(\tfrac{t}{\beta}\big)
\right)
, &
\qquad \text{if } t\in [-\beta/3, \beta/3];\\
a_1 \, (t+\frac{\beta}{3}) + b_1 , &
\qquad \text{if } t\in (-\infty, -\beta/3); \\
a_2 \, (t-\frac{\beta}{3}) + b_2 , &
\qquad \text{if } t\in (\beta/3, \infty),
\end{cases}
\end{equation}
where $a_1, b_1, a_2, b_2$ are chosen so that $\EmbTan{\beta}_\phi$ is $C^1$ at $-\beta/3$ and $\beta/3$, namely
\begin{align}
a_1 = 
\left.
\frac{\d}{\d s}\left(
\tfrac{\beta}{\pi} \tan\left(\pi \,
\GProj_\phi\big(\tfrac{s}{\beta}\big)
\right)
\right)
\right|_{s = -\beta/3},
\qquad
&
b_1 = 
\left.
\tfrac{\beta}{\pi} \tan\left(\pi \,
\GProj_\phi\big(\tfrac{s}{\beta}\big)
\right)
\right|_{s = -\beta/3},
\\
a_2 = 
\left.
\frac{\d}{\d s}\left(
\tfrac{\beta}{\pi} \tan\left(\pi \,
\GProj_\phi\big(\tfrac{s}{\beta}\big)
\right)
\right)
\right|_{s = \beta/3},
\qquad
&
b_2 = 
\left.
\tfrac{\beta}{\pi} \tan\left(\pi \,
\GProj_\phi\big(\tfrac{s}{\beta}\big)
\right)
\right|_{s = \beta/3}.
\end{align}

Note that
\begin{equation}\label{eqObsToObs_0}
\obs{\phi}{s}{t}
=
\beta \obs{\,\EmbTan{\beta}_{\phi}}{s \beta}{t \beta}_0,
\end{equation}
when $s, t \in [-1/3, 1/3]$.

\subsection{Convergence of correlation functions}
\label{sectThrLocalProofCharacter}

\begin{prp}\label{prpCorrConv}
For $N\geq 0$ let $\big\{\obs{f}{s_{j}}{t_j}_0\big\}_{j=1}^{N}$ be a set of non-interlaced observables, and $\{l_j\}_{j=1}^N$ be a set of positive integers. 
Assume that $\sigma, \beta\to\infty$ with $\sigma^2/\beta \to \varkappa^2$.
Then 
\begin{equation}
\lim_{\sigma, \beta\to \infty}
\PartF_{\sigma^2}^{-1}
\int\limits_{\CZPlus(\R)} \prod_{j=1}^N \obs{f}{s_j}{t_j}_{0}^{l_j}\d\EmbTan{\beta}_{\sharp}\FMeas{\sigma^2}(f)
=
\int\limits_{\CZPlus(\R)} \prod_{j=1}^N \obs{f}{s_j}{t_j}_0^{l_j}\d\FMInf{\varkappa^2}(f).
\end{equation}
\end{prp}
\begin{proof}
From \eqref{eqObsToObs_0} and the explicit formula for the correlation functions of $\FMeas{\sigma^2}$ (see Theorem~\ref{thrMainCorrelations_corr}),
\begin{multline}
\int\limits_{\CZPlus(\R)} \prod_{j=1}^N \obs{f}{s_j}{t_j}_{0}^{l_j}\d\EmbTan{\beta}_{\sharp}\FMeas{\sigma^2}(f)
=
\int_{\R_+^{N+1}} 
\prod_{j=1}^N \frac{\Gamma\left(\dfrac{l_j}{2}\pm i w_1(j) \pm i w_2(j)\right)}{2\pi^2 \, \Gamma(l_j)}
\cdot\left(\frac{\sigma^2}{2\beta}\right)^{l_j}\\
\times
\prod_{m=0}^{N}\exp\left(-\frac{\tau_m \sigma^2}{2\beta}\cdot k_m^2\right)  \sinh(2\pi k_m) \, 2 k_m \d k_m,
\end{multline}
whenever $\beta > 10 \max\{|s_j|, |t_j|\}_{j=1}^N$.
In the formula above we use the diagrammatic language for chords $\{(s_j, t_j)\}_{j=1}^N$ on the circle of size $\beta$, meaning that $\sum_{m=0}^{N}\tau_m = \beta$.
In particular, in the $\sigma, \beta\to\infty$ limit we have that $\{\tau_m\}_{m=1}^N$ stay fixed, while $\tau_0/\beta \to 1$.
Notice that for any continuous and bounded $H:\R\to \R$,
\begin{equation}
\lim_{\sigma, \beta\to \infty} \PartF_{\sigma^2}^{-1}\int H(k_0) \exp\left(-\frac{\tau_0 \sigma^2}{2\beta}\cdot k_0^2\right)  \sinh(2\pi k_0) \, 2 k_0 \d k_0
= H(0).
\end{equation}
Together with Theorem~\ref{thrLineCorrelations} this finishes the proof.
\end{proof}

\subsection{Tightness}
\label{sectThrLocalProofTight}

The aim of this section is to prove Proposition~\ref{prpTightnessLoc}.
We do this by showing that \eqref{eqObsHolderLine} holds with large probability with respect to $\mathcal{Z}_{\sigma^2}^{-1}\cdot \EmbTan{\beta}_{\sharp}\FMeas{\sigma^2}$ when $\sigma$ and $\beta$ are large.
In Section~\ref{sectObsDevEst} we demonstrate that \eqref{eqObsHolderLine} holds with large probability for just one observable.
In Section~\ref{sectObsInterpol} we show that, with large probability, \eqref{eqObsHolderLine} holds simultaneously for all $s$ and $t$ close enough.

\subsubsection{Observables deviation estimates}
\label{sectObsDevEst}
\begin{prp}\label{prpBoundExpMomentsLine}
There exists $C>1$ such that for any $\sigma>0$, any $s \neq t\in \T$,
and any $\lambda\in [-C^{-1}, C^{-1}]$, we have
\begin{equation}\label{eqMomentGenFuncBoundLine}
\int \exp\left\{\frac{2\lambda}{\sigma^2} \left(\obs{\phi}{s}{t} -\frac{1}{|t- s|}\right)\right\}\d \FMeas{\sigma^2}(\phi)
<\left(1-C|\lambda|\right)^{-1/2}
\exp\left\{\frac{C \lambda^2}{4|t-s|\sigma^2}\right\}
\mathcal{Z}\Big((1-|t-s|)\sigma^2\Big).
\end{equation}
\end{prp}
\begin{proof}
By rotational invariance of $\d\FMeas{\sigma^2}$, without loss of generality we can assume that $s=0$.

In \citep{LosevLDP} it was proved that for some $C_1>1$ and all $\lambda \in [-C_1^{-1}, C_1^{-1}]$,
\begin{multline}
\int \exp\left\{\frac{2\lambda}{\sigma^2} \left(\obs{\phi}{0}{t} -\frac{1}{t}\right)\right\}\d \FMeas{\sigma^2}(\phi) \\
\leq
\frac{\sqrt{2}}{\sqrt{\pi t} \, \sigma}
\int_{\R_+^2}
\exp\left\{ -\frac{\alpha^2}{2t\sigma^2}
+ \frac{C_1|\lambda|(\alpha^2+ |\lambda|)}{2t\sigma^2}
\right\}
\exp\left(-\frac{(1-t)\sigma^2}{2}\cdot k_2^2\right)
\sinh(2\pi k_2)\, 2k_2 \d k_2
\d \alpha.
\end{multline}
For the reader's convenience, we recall the proof of this estimate in the Appendix, see Proposition~\ref{prpAppExpEst_LDP}.
Integrating over $\alpha$ and $k_2$,
\begin{multline}
\frac{\sqrt{2}}{\sqrt{\pi t} \, \sigma}
\int_{\R_+^2}
\exp\left\{ -\frac{\alpha^2}{2t\sigma^2}
+ \frac{C_1|\lambda|(\alpha^2+ |\lambda|)}{2t\sigma^2}
\right\}
\exp\left(-\frac{(1-t)\sigma^2}{2}\cdot k_2^2\right)
\sinh(2\pi k_2)\, 2k_2 \d k_2
\d \alpha\\
\leq
\left(1-C_1|\lambda|\right)^{-1/2}
\exp\left\{\frac{C_1 \lambda^2}{4t\sigma^2}\right\}
\mathcal{Z}\Big((1-t)\sigma^2\Big),
\end{multline}
which finishes the proof.

\end{proof}

The following corollary allows to control deviations of observables, which appear in the definition of the Local H\"{o}lder Condition for Observables (see Definition~\ref{defnHolderLine}).
\begin{crl}\label{crlObsTailBoundLine}
There exists $C>1$ such that for any $M, \beta, \sigma>0$, and $s\neq t\in \R$ with $|t-s|
<\min\{1, \, \beta/2\}
$,
\begin{multline}
\FMeas{\sigma^2} \left(\phi\in \Diff^1(\T)/\SL(2, \R)\,:\, \left|\, \obs{ \tfrac{\beta}{\pi}\tan \left(\tfrac{\pi}{\beta}\Emb{\beta}(\phi)\right)}{s}{t}_{0}-\frac{1}{|t-s|} \, \right|>M\, |t-s|^{-\tfrac{1}{2}}\right)\\
<
C \cdot
\mathcal{Z}\left(\left(1-\frac{|t-s|}{\beta}\right)\sigma^2\right)
\cdot
\exp\left(\frac{\beta}{\sigma^2}\left(C -C^{-1}M\right)\right).
\end{multline}
\end{crl}
In this corollary one should think of $M$ being very large, $s$ and $t$ close, and $\sigma$ and $\beta$ large with $\sigma^2/\beta \approx \varkappa^2 \in (0, \infty)$.
In this regime 
\begin{equation}
\mathcal{Z}\left(\left(1-\frac{|t-s|}{\beta}\right)\sigma^2\right) \sim \mathcal{Z}(\sigma^2).
\end{equation}

\begin{proof}[Proof of Corollary~\ref{crlObsTailBoundLine}]
From identity
\begin{equation}
\beta \, \obs{\, \EmbTan{\beta}(\phi)}{s}{t}_{0}
=
\obs{\phi}{s/\beta}{t/\beta},
\end{equation}
and Proposition~\ref{prpBoundExpMomentsLine}, we get that for some $C>1$ and $\lambda \in [-C^{-1}, C^{-1}]$,
\begin{multline}
\int \exp\left\{2\lambda \, \frac{\beta}{\sigma^2} \left(\obs{ \tfrac{\beta}{\pi}\tan \left(\tfrac{\pi}{\beta}\Emb{\beta}_{\phi}\right)}{s}{t}_{0} -\frac{1}{|t- s|}\right)\right\}\d \FMeas{\sigma^2}(\phi)\\
<\left(1-C\lambda\right)^{-1/2}
\exp\left\{\frac{C \lambda^2}{4|t-s|}\cdot\frac{\beta}{\sigma^2}\right\}
\mathcal{Z}\left(\left(1-\frac{|t-s|}{\beta}\right)\sigma^2\right).
\end{multline}
Now taking $\lambda = 10^{-1} \, C^{-1}|t-s|^{\frac{1}{2}}$ and using Markov's Inequality for moment generating functions gives the desired result.
\end{proof}

\subsubsection{Observables interpolation}
\label{sectObsInterpol}

\begin{lmm}\label{lmmObsAlgEquality}
Let $\tau_1, \tau_2, \tau_3, \tau_4$ be distinct points that lie on $\R$ in either increasing or decreasing order. 
Then for any $f \in \CZPlus(\R)$
\begin{equation}\label{lmmObsAlgEqualityEq}
\frac{1}{\obs{f}{\tau_1}{\tau_3}_{0}\obs{f}{\tau_2}{\tau_4}_{0}} = 
\frac{1}{\obs{f}{\tau_1}{\tau_2}_{0}\obs{f}{\tau_3}{\tau_4}_{0}}
+\frac{1}{\obs{f}{\tau_1}{\tau_4}_{0}\obs{f}{\tau_2}{\tau_3}_{0}}.
\end{equation}
\end{lmm}
\begin{proof}
Without loss of generality we can assume that $\{\tau_j\}_j$ is an increasing sequence.
Notice that if $u, v, w$ lie on $\R$ in increasing order, then
\begin{multline}
\frac{\obs{f}{u}{v}_{0}\obs{f}{u}{w}_{0}}{\obs{f}{v}{w}_{0}} 
= \frac{f'(u) \, |f(w)-f(v)|}{|f(v)-f(u)| \cdot |f(w)-f(u)|} \\
= f'(u)\Biggl(\frac{1}{|f(v)-f(u)|}-  \frac{1}{|f(w)-f(u)|} 
\Biggl).
\end{multline}
Therefore,
\begin{equation}
\frac{\obs{f}{\tau_1}{\tau_2}_{0} \obs{f}{\tau_1}{\tau_3}_{0}}{\obs{f}{\tau_2}{\tau_3}_{0}} 
+\frac{\obs{f}{\tau_1}{\tau_4}_{0}\obs{f}{\tau_1}{\tau_3}_{0}}{\obs{f}{\tau_3}{\tau_4}_{0}} 
=
\frac{\obs{f}{\tau_1}{\tau_2}_{0}\obs{f}{\tau_1}{\tau_4}_{0}}{\obs{f}{\tau_2}{\tau_4}_{0}},
\end{equation}
which is equivalent to \eqref{lmmObsAlgEqualityEq}.
\end{proof}

\begin{lmm}\label{lmmObsIneq}
Assume that $s, s_1, t_1, t, t_2$ are points that lie on $\R$ in either increasing or decreasing order. 
Then for any $f \in \CZPlus(\R)$
\begin{equation}
\frac{\obs{f}{s}{t_1}_{0}\obs{f}{s_1}{t}_{0}}{\obs{f}{s_1}{t_1}_{0}}
<\obs{f}{s}{t}_{0}
< \frac{\obs{f}{s}{t_2}_{0} \obs{f}{s_1}{t}_{0}}{\obs{f}{s_1}{t_2}_{0}}.
\end{equation}
\end{lmm}
\begin{proof}

It follows from Lemma~\ref{lmmObsAlgEquality} that
\begin{multline}
\frac{1}{
\obs{f}{s}{t}_{0}} 
= \obs{f}{s_1}{t_2}_{0}\Biggl(\frac{1}{\obs{f}{s}{t_2}_{0}\obs{f}{s_1}{t}_{0}}
+\frac{1}{\obs{f}{s}{s_1}_{0}\obs{f}{t}{t_2}_{0}}\Biggl)\\
>\frac{\obs{f}{s_1}{t_2}_{0}}{\obs{f}{s}{t_2}_{0}\obs{f}{s_1}{t}_{0}},
\end{multline}
and
\begin{multline}
\frac{1}{\obs{f}{s}{t}_{0}}
=
\obs{f}{s_1}{t_1}_{0}
\Biggl(\frac{1}{\obs{f}{s}{t_1}_{0}\obs{f}{s_1}{t}_{0}}
-\frac{1}{\obs{f}{s}{s_1}_{0}\obs{f}{t_1}{t}_{0}}
\Biggl) \\
< \frac{\obs{f}{s_1}{t_1}_{0}}{\obs{f}{s}{t_1}_{0}\obs{f}{s_1}{t}_{0}},
\end{multline}
which finishes the proof.
\end{proof}

\begin{lmm}\label{lmmObsInterpolationDyadic}
Assume that $A, \delta>0$ and $x\in\R$.
Denote $\Dyadic{n} = \left\{k/2^n \, : \, k\in \Z \right\} \cap (x-\delta, x+\delta)$. 
Let $\Aset_n$ be the event
\begin{multline}
\Aset_n = \Biggl\{f \in \CZPlus(\R) \, :\, \Big| \, \log \left[\obs{f}{s}{t}_{0}\right]+\log |t-s|\, \Big| < A \, |t-s|^{\alpha}, \\
\text{for all } s,t\in \Dyadic{n} \text{ such that } |t-s| \leq 100 \cdot 2^{-n/2} 
\Biggl\}.
\end{multline}
Then there exists $B=B(A, \delta, \alpha)>0$ such that for any $N>0$, there exists $\eps=\eps(N, A, \delta, \alpha)>0$ for which
\begin{equation}
\ObsHolLoc{\alpha} \! \left(x, \eps, B \right) \supset (\cap_{n \geq N} \, \Aset_n).
\end{equation}
\end{lmm}
\begin{proof}

In this proof $C_1, C_2, \ldots$ denote constants that depend only on $A$, $\delta$, and $\alpha$ (in particular, they do not depend on $f$).

\textbf{Step 1.} First, we show that there exists a large $M_1 = M_1(A, \delta, \alpha)>0$ and a small $\eps_1 = \eps_1 (A, \delta, \alpha) > 0$ (both to be specified after $C_1, C_2, C_3$) such that the following holds:
for any $f \in \cap_{k \geq N} \, \Aset_k$, any $m\geq n\geq N$, any $s\in \Dyadic{n} \cap (x-\eps_1, x+\eps_1)$ and any $t \in \Dyadic{m} \cap (x-\eps_1, x+\eps_1)$ with $|t-s|\leq 10 \cdot 2^{-n}$, we have
\begin{equation}\label{eqLogObsApproxInduction}
 \left|\log \left[\obs{f}{s}{t}_{0}\right] + \log|t-s|\right|
 <   M_1  \, |t-s|^{\alpha}.
\end{equation}

We prove this by reverse induction on $n$. We also assume that $t\notin\Dyadic{m-1}$.

\textit{Base Case:} $n=m$.
This follows from the fact that $f\in \Aset_m$.

\textit{Induction Step:} $n+1 \mapsto n$.

If $|t-s|\leq 10\cdot 2^{-n-1}$, then we just apply the induction hypothesis, since $t\in \Dyadic{n} \subset \Dyadic{n+1}$. 
Therefore, now we can assume $|t-s|> 10\cdot 2^{-n-1}$.

Let $s_1 \in \Dyadic{n+1}$ be the second point between $t$ and $s$, counting in the direction from $t$ to $s$.
In other words, $s_1$ lies between $t$ and $s$, with $2^{-n-1} \leq |s_1-t| < 2\cdot 2^{-n-1} $.
Notice that, since $s, t \in (x-\eps_1, x+ \eps_1)$, we also automatically have that $s_1\in (x-\eps_1, x+ \eps_1)$, and so we can apply the induction hypothesis to $s_1 \in \Dyadic{n+1} \cap  (x-\eps_1, x+ \eps_1)$ and $t\in \Dyadic{m} \cap (x-\eps_1, x+\eps_1)$.

If $2n< m$, we let $t_1, t_2\in \Dyadic{2n}$ be two points such that $t$ lies on the dyadic interval of size $2^{-2n}$ between them. 
Moreover, we assume that $t_1$ lies between $s_1$ and $t$. 
In other words, $s, s_1, t_1, t, t_2$ lie on $\R$ either in increasing or decreasing order.
If $2n\geq  m$, we just put $t_1 = t_2 = t$.
Now we can apply Lemma~\ref{lmmObsIneq}.

\medskip

First, we give an upper bound for $\log \left[\obs{f}{s}{t}_{0}\right]$.
From Lemma \ref{lmmObsIneq} we deduce that
\begin{equation}\label{eqLogObsBoundAboveInduction}
\log \left[\obs{f}{s}{t}_{0}\right]
\leq 
\log \left[\obs{f}{s_1}{t}_{0}\right]
+\log \left[\obs{f}{s}{t_2}_{0}\right]
- \log \left[\obs{f}{s_1}{t_2}_{0}\right].
\end{equation}

Using induction hypothesis for $s_1$ and $t$, and the fact that $|t-s_1|\leq 2\cdot 2^{-n-1}$, we get
\begin{equation}\label{eqLogObsBoundAboveInductionFirstTerm}
\log \left[\obs{f}{s_1}{t}_{0}\right]
\leq - \log|t-s_1| 
+ M_1 \, |t-s_1|^{\alpha}
.
\end{equation}
Moreover, since $f \in \Aset_{2n}$, $s, t_2\in\Dyadic{2n}$ and 
$|t_2-s|\leq |t-s|+|t_2-t_1|\leq 10\cdot 2^{-n} +2^{-2n} < 100\cdot 2^{-n}$, we have
\begin{equation}
\log \left[\obs{f}{s}{t_2}_{0}\right]
\leq
-\log|t_2-s| 
+  A\, |t_2-s|^{\alpha}.
\end{equation}
Combining this with the fact that $10 \cdot 2^{-n-1}<|t-s|<10 \cdot 2^{-n}$, and inequality $|t_2-t|\leq|t_2-t_1|\leq 2^{-2n}$ we obtain
\begin{equation}\label{eqLogObsBoundAboveInductionSecondTerm}
\log \left[\obs{f}{s}{t_2}_{0}\right]
\leq
-\log|t-s|
+ C_1 \, 2^{-\alpha n}
.
\end{equation}
Similarly, since $f \in \Aset_{2n}$, $s_1, t_2\in\Dyadic{2n}$ and 
$|t_2 - s_1|\leq |t - s_1| + |t_1-t_2| \leq 2\cdot 2^{-n-1} + 2^{-2n} < 100\cdot 2^{-n}$, we get
\begin{equation}
\log \left[\obs{f}{s_1}{t_2}_{0}\right] 
\geq 
-\log|t_2-s_1|
- A\, |t_2-s_1|^{\alpha}.
\end{equation}
Combining this with the inequalities $2^{-n-2} \leq |t-s_1|\leq 2^{-n-1}$ and $|t_2-t|\leq |t_2-t_1| \leq 2^{-2n}$, we obtain
\begin{equation}\label{eqLogObsBoundAboveInductionThirdTerm}
\log \left[\obs{f}{s_1}{t_2}_{0}\right] 
\geq
-\log|t-s_1|
- C_2 \, 2^{-\alpha n}.
\end{equation}

Using \eqref{eqLogObsBoundAboveInduction}, \eqref{eqLogObsBoundAboveInductionFirstTerm}, \eqref{eqLogObsBoundAboveInductionSecondTerm}, and \eqref{eqLogObsBoundAboveInductionThirdTerm}, we get
\begin{equation}\label{eqLogObsBoundAboveInductionLastStep}
\log \left[\obs{f}{s}{t}_{0}\right]
\leq 
-\log|t-s|  
+ M_1 \, |t-s_1|^{\alpha}
+ C_3 \, 2^{-\alpha n}
.
\end{equation}
Since $|t-s|>10 \cdot 2^{-n-1}$, and $|t-s_1|\leq 2\cdot 2^{-n-1}$, we deduce that for $M_1$ sufficiently large (depending on $C_3$ and $\alpha$)
\begin{equation}
M_1 \, |t-s_1|^{\alpha}
+ C_3 \, 2^{-\alpha n}
\leq
 M_1 \, |t-s|^{\alpha},
\end{equation}
and so
\begin{equation}
\log \left[\obs{f}{s}{t}_{0}\right]
\leq 
-\log|t-s|
+ M_1 \, |t-s|^{\alpha}.
\end{equation}

\medskip

Second, analogously to the upper bound for $\log \left[\obs{f}{s}{t}_{0}\right]$, we can bound it from below using Lemma~\ref{lmmObsIneq} by
\begin{equation}
\log \left[\obs{f}{s}{t}_{0}\right] 
\geq 
-\log|t-s| 
- M_1\, |t-s|^{\alpha}.
\end{equation}
This finishes the proof of the induction step. Thus, we have proved \eqref{eqLogObsApproxInduction}.

\medskip

\textbf{Step 2.} Now, we show that there exists a large $M_2=M_2(A, \delta, \alpha)>0$ and a small $\eps_2=\eps_2(N, A, \delta, \alpha)>0$ (both to be specified after $M_1$, $\eps_1$, and $C_1, \ldots, C_9$) such that the following holds:
for any $f \in \cap_{k \geq N} \Aset_k$ and any $s, t \in \cap_{k \geq N}  \Dyadic{k}\cup (x-\eps_2, x+\eps_2) $, we have
\begin{equation}\label{eqLogObsApproxUnifStep2Goal}
 \left|\obs{f}{s}{t}_{0}-\frac{1}{|t-s|}\right|
 < M_2 \, |t-s|^{\alpha-1},
\end{equation}
which implies the desired result with $B=M_2$ and $\eps = \eps_2$.

Let $n\geq N$ be such that $2^{-n+2} < |t-s|\leq 2^{-n+3}$.
Also, let $m\geq n$ be such that $s, t\in \Dyadic{m}$. 
Since $ |t-s|> 4\cdot 2^{-n}$, there exist $r_1, r_2\in \Dyadic{n}$ between $s$ and $t$ such that $|r_2-r_1| = 2^{-n}$, $ |r_1-s| \geq 2^{-n}$, $|t-r_2|\geq 2^{-n}$,
and such that points $s, r_1, r_2, t$ lie on $\R$ either in decreasing or in increasing order.
Applying Lemma~\ref{lmmObsAlgEquality}, we get
\begin{equation}\label{eqObsApproxEquality}
\frac{1}{\obs{f}{s}{t}_{0}} = \obs{f}{r_1}{r_2}_{0}\Biggl(\frac{1}{\obs{f}{s}{r_2}_{0}\obs{f}{r_1}{t}_{0}}
-\frac{1}{\obs{f}{s}{r_1}_{0}\obs{f}{r_2}{t}_{0}}\Biggl)
\end{equation}
Since $|t-s|\leq 8\cdot 2^{-n} < 10\cdot 2^{-n} $, we can apply \eqref{eqLogObsApproxInduction} to estimate all terms in the right-hand side of equation above to get (with $C_4, C_5$ which may depend on $M_1$),
\begin{equation}\label{eqObsApproxEqualityFirstTermBound}
\left|\frac{\obs{f}{r_1}{r_2}_{0}}{\obs{f}{s}{r_2}_{0}\obs{f}{r_1}{t}_{0}}
-\frac{|r_2-s| \cdot |t-r_1|}{|r_2-r_1|}\right|\\
\leq C_4 \frac{|r_2-s| \cdot |t-r_1|}{|r_2-r_1|} \, 2^{-\alpha n}
\leq C_5\, 2^{-(\alpha+1)n}.
\end{equation}
Similarly,
\begin{equation}\label{eqObsApproxEqualitySecondTermBound}
\left|\frac{\obs{f}{r_1}{r_2}_{0}}{\obs{f}{s}{r_1}_{0}\obs{f}{r_2}{t}_{0}}
-\frac{|r_1-s|\cdot |t-r_2|}{|r_2-r_1|}\right|
\leq C_6 \, \frac{|r_1-s|\cdot |t-r_2|}{|r_2-r_1|} \, 2^{-\alpha n}
\leq C_7 \, 2^{-(\alpha+1)n}.
\end{equation}

Notice that
\begin{equation}
\frac{|r_2-s|\cdot|t-r_1|}{|r_2-r_1|} 
- \frac{|r_1-s| \cdot |t-r_2|}{|r_2-r_1|}
=
|t-s|.
\end{equation}
Therefore, combining this with  \eqref{eqObsApproxEquality}, \eqref{eqObsApproxEqualityFirstTermBound}, and \eqref{eqObsApproxEqualitySecondTermBound}, we obtain
\begin{equation}
\left|\frac{1}{\obs{f}{s}{t}_{0}}-|t-s|\right| 
\leq
C_8 \, 2^{-(\alpha+1)n} 
\leq C_9\, |t-s|^{\alpha+1}.
\end{equation}
Thus, for large enough $M_2$,
\begin{equation}
\left|\obs{f}{s}{t}_{0}-\frac{1}{|t-s|}\right| 
\leq M_2 \, |t-s|^{\alpha-1}.
\end{equation}
\end{proof}

\subsection{Tightness proof}
The following proposition, together with Corollary~\ref{crlCpmctFromObs}, implies tightness of measures $\mathcal{Z}_{\sigma^2}^{-1} \cdot \d \EmbTan{\beta}_{\sharp} \FMeas{\sigma^2}$ in the setting of Theorem~\ref{thrLocalLimit}.

To formulate the following proposition, we use the inner measure $\left(\EmbTan{\beta}_{\sharp} \FMeas{\sigma^2}\right)_{*}$, which is defined for any $\Aset\subset\CZPlus(\R)$ by
\begin{equation}
\left(\EmbTan{\beta}_{\sharp} \FMeas{\sigma^2}\right)_{*} (\Aset) 
=
\sup\left\{
\EmbTan{\beta}_{\sharp} \FMeas{\sigma^2} (S): \, S\subset \CZPlus(\R) \text{ is Borel measurable and } S\subset \Aset
\right\}.
\end{equation}
\begin{prp}\label{prpTightnessLoc}
Fix $\Lambda>0$ and $\alpha\in (0, 1/2)$.
Then for any $\eps>0$ there exist positive $\{A_x\}_{x\in\R}, \{\delta_x\}_{x\in\R}$ such that for any $\sigma, \beta>1$ with $\sigma^2/\beta<\Lambda$ we have
\begin{equation}
\mathcal{Z}_{\sigma^2}^{-1} \cdot  
\left(\EmbTan{\beta}_{\sharp} \FMeas{\sigma^2}\right)_{*}
\Big(
\cap_{x\in \R} \ObsHolLoc{\alpha} \! \left(x, \delta_x, A_x \right) 
\Big) > 1- \eps.
\end{equation}
\end{prp}

\begin{proof}
It is sufficient to show that for any $\eta\in (0,1)$
there exist $B_{\eta}=B(\eta, \Lambda, \alpha)>0$ and $\delta_{\eta} = \delta(\eta, \Lambda, \alpha)>0$ such that
\begin{equation}\label{eqPrpTightProof1}
\mathcal{Z}_{\sigma^2}^{-1} \cdot  
\left(\EmbTan{\beta}_{\sharp} \FMeas{\sigma^2}\right)_{*}
\Big(
\cap_{x\in (-\eta^{-1}, \eta^{-1})} \ObsHolLoc{\alpha} \! \left(x, \delta_{\eta}, B_\eta \right) 
\Big) > 1- \eta
\end{equation}
for any $\sigma$ and $\beta$ as in the statement of the Proposition.

Indeed, if \eqref{eqPrpTightProof1} holds, then taking $\eta = 2^{-m}$ and summing the inequality above (and taking the intersection of the events on the left-hand side) over $m>M$ for some large $M>0$, we get the desired result.

Now we prove \eqref{eqPrpTightProof1}.
Essentially, there are two cases.
If $\beta/3$ is larger than $\eta^{-1}$, then for any $\phi$ we have that $\EmbTan{\beta}(\phi)$ agrees with $\tfrac{\beta}{\pi} \tan\left(\tfrac{\pi}{\beta} \, \Emb{\beta}(\phi) \right)$ on $[-\eta^{-1}, \eta^{-1}]$ and \eqref{eqPrpTightProof1} follows from Lemma~\ref{lmmObsInterpolationDyadic} and Corollary~\ref{crlObsTailBoundLine}.
If $\beta/3$ is smaller than $\eta^{-1}$, then the combination of Lemma~\ref{lmmObsInterpolationDyadic} and Corollary~\ref{crlObsTailBoundLine} yields the Local H\"{o}lder Condition for Observables on $[-\beta/3, \beta/3]$ (recall that  $\EmbTan{\beta}(\phi)$ changes the definition at $\{-\beta/3, \beta/3\}$).
After that we have to extend this Local H\"{o}lder Condition for Observables to $[-\eta^{-1}, \eta^{-1}]$ by using that this condition is equivalent to local H\"{o}lder condition for the logarithm of the derivative (Theorem~\ref{thrLocObsHolderLine}).
Since  $\EmbTan{\beta}(\phi)$ is always linear outside $[-\beta/3, \beta/3]$, the local H\"{o}lder condition for the logarithm of the derivative immediately extends from $[-\beta/3, \beta/3]$ to $[-\eta^{-1}, \eta^{-1}]$.
Then, using Theorem~\ref{thrLocObsHolderLine} again, we convert  the local H\"{o}lder condition for the logarithm of the derivative back to the Local H\"{o}lder Condition for Observables.

For technical reasons we need to consider a new embedding $\widehat{\Emb{\beta}}:\, \Diff^1(\T)/\SL(2, \R) \to \CZPlus(\R) $ given by
\begin{equation}
\phi \mapsto \widehat{\Emb{\beta}}(\phi)(t) \coloneqq \widehat{\Emb{\beta}_\phi} (t)
\coloneqq 
\begin{cases}
\tfrac{\beta}{\pi} \tan\left(\tfrac{\pi}{\beta} \,
\Emb{\beta}_\phi (t)
\right) 
=
\tfrac{\beta}{\pi} \tan\left(\pi \,
\GProj_\phi\big(\tfrac{t}{\beta}\big)
\right)
, &
\qquad \text{if } t\in [-2\beta/5,\, 2\beta/5];\\
c_1 \, t + d_1 , &
\qquad \text{if } t\in (-\infty, \, -2\beta/5); \\
c_2 \, t + d_2 , &
\qquad \text{if } t\in (2\beta/5, \, \infty),
\end{cases}
\end{equation}
where $c_1, d_1, c_2, d_2$ are chosen so that $\widehat{\Emb{\beta}}$ is $C^1$ at $-2\beta/5$ and $2\beta/5$.
The point of $\widehat{\Emb{\beta}}$ is that it coincides with $\tfrac{\beta}{\pi} \tan\left(\tfrac{\pi}{\beta} \,
\Emb{\beta}(\phi)\right) $ on a slightly larger interval than $\EmbTan{\beta}$ does, since $2\beta/5>\beta/3$.

Now we fix some $\beta>1$.
For a positive integer $n$ denote 
\begin{equation}
\Dyadic{n} = \left\{k/2^n \, : \, k\in \Z \right\} \cap \big(-\min\{2\beta/5,\, 6\eta^{-1}/5\},\, \min\{2\beta/5,\, 6\eta^{-1}/5\}\big).
\end{equation}
Consider the event
\begin{multline}
\Aset_n = \Biggl\{f \in \CZPlus(\R) \, :\, \Big| \, \log \left[\obs{f}{s}{t}_{0}\right]+\log |t-s| \, \Big| < |t-s|^{\alpha}, \\
\text{for all } s,t\in \Dyadic{n} \text{ such that } |t-s| \leq 100 \cdot 2^{-n/2} 
\Biggl\}.
\end{multline}

From Corollary~\ref{crlObsTailBoundLine} and the explicit formula for the partition function $\mathcal{Z}$ (see Proposition~\ref{prpMainPartFunct}) we have that for some $C=C(\eta, \Lambda, \alpha)>0$,
\begin{equation}
\mathcal{Z}_{\sigma^2}^{-1} \cdot 
\widehat{\Emb{\beta}_\sharp}\FMeas{\sigma^2}
 \left(\Aset_n\right)
\geq 
1-
C\, \eta^{-1} \, 2^{3n/2} \exp\left(- C^{-1} \, 2^{n(1-2\alpha)/10} \right).
\end{equation}
Thus, there exists $N = N(\eta, \Lambda, \alpha) > 0$ such that
\begin{equation}
\mathcal{Z}_{\sigma^2}^{-1} \cdot 
\widehat{\Emb{\beta}_\sharp}\FMeas{\sigma^2} \Big(\cap_{n=N}^{\infty} \Aset_n\Big) > 1-\eta.
\end{equation}
From Lemma~\ref{lmmObsInterpolationDyadic} we get that (this is the place where we need to use that $\widehat{\Emb{\beta}_\sharp}$ extends well slightly outside of $[-\beta/3, \beta/3]$, as we shrink the interval; here we choose to dilate the interval by a factor of $5/6$)
\begin{equation}\label{eqPrpTightProof2}
\mathcal{Z}_{\sigma^2}^{-1} \cdot  
\left( \widehat{\Emb{\beta}_\sharp} \FMeas{\sigma^2}\right)_{*}
\Big(
\cap_{x\in [-\min\{\beta/3,\, \eta^{-1}\},\, \min\{\beta/3,\, \eta^{-1}\}]} \ObsHolLoc{\alpha} \! \left(x, \delta_{\eta}, B_\eta \right) 
\Big) > 1- \eta.
\end{equation}
We now consider two cases.

\textbf{Case 1.} Suppose $\beta \geq 3\eta^{-1} + \delta_\eta$.
Then
\begin{equation}
\bigcap_{x\in [-\min\{\beta/3,\, \eta^{-1}\},\, \min\{\beta/3,\, \eta^{-1}\}]} \ObsHolLoc{\alpha} \! \left(x, \delta_{\eta}, B_\eta \right) 
\subset
\bigcap_{x\in [-\eta^{-1},\, \eta^{-1}]} \ObsHolLoc{\alpha} \! \left(x, \delta_{\eta}/10, B_\eta \right).
\end{equation}
Thus, \eqref{eqPrpTightProof1} holds (with $\delta_\eta/10$ in place of $\delta_\eta$) because for any $\phi\in \Diff^1(\T)/\SL(2, \R)$ we have that $\widehat{\Emb{\beta}}(\phi) $ agrees with $\EmbTan{\beta}(\phi)$ on $[-\beta/3,\, \beta/3]$.

\textbf{Case 2.}
Now we assume that $\beta < 3\eta^{-1} + \delta_\eta$.

From Theorem~\ref{thrLocObsHolderLine} it follows that there exist $\delta'_{\eta}>0$ and $B'_{\eta}>0$ such that
\begin{equation}\label{eqPrpTightProof3}
\bigcap_{x\in [- \min\{\beta/3,\, \eta^{-1}\},\,  \min\{\beta/3,\, \eta^{-1}\}]} 
\ObsHolLoc{\alpha} \! \left(x, \delta_{\eta}, B_\eta\right)
\subset
\mathrm{C}_{\beta, \eta},
\end{equation}
where
\begin{equation}
\mathrm{C}_{\beta, \eta}
=
\Biggl\{f \in \CZPlus(\R) \, :\quad
\begin{matrix}
\forall s\neq t \in [- \beta/3,\, \beta/3] \text{ with } |t-s|<  \delta'_{\eta}\\
\text{we have that } |\log f'(t)-\log f'(s)| <B'_{\eta}|t-s|^{\alpha}
\end{matrix}
\Biggl\}.
\end{equation}
Combining this with \eqref{eqPrpTightProof2} and \eqref{eqPrpTightProof3} we get
\begin{equation}
\mathcal{Z}_{\sigma^2}^{-1} \cdot  \widehat{\Emb{\beta}_\sharp} \FMeas{\sigma^2}\Big(\mathrm{C}_{\beta, \eta}  
\Big) > 1- \eta.
\end{equation}
Since for any $\phi\in \Diff^1(\T)/\SL(2, \R)$ we have that $\widehat{\Emb{\beta}}(\phi) $ agrees with $\EmbTan{\beta}(\phi)$ on $[-\beta/3,\, \beta/3]$,
\begin{equation}
\mathcal{Z}_{\sigma^2}^{-1} \cdot  \EmbTan{\beta}_\sharp \FMeas{\sigma^2}\Big(\mathrm{C}_{\beta, \eta}  
\Big) > 1- \eta.
\end{equation}
Moreover, since $\EmbTan{\beta}(\phi)$ is linear outside of $[-\beta/3,\, \beta/3]$, we can extend the local H\"{o}lder condition for $\log f'$ to the whole line,
\begin{multline}
\mathcal{Z}_{\sigma^2}^{-1} \cdot  \EmbTan{\beta}_\sharp \FMeas{\sigma^2}\Big(\mathrm{C}_{\beta, \eta}  
\Big) \\
=
\mathcal{Z}_{\sigma^2}^{-1} \cdot  \EmbTan{\beta}_\sharp \FMeas{\sigma^2}
\Biggl\{f \in \CZPlus(\R) \, :\quad
\begin{matrix}
\forall s\neq t \in \R \text{ with } |t-s|<  \delta'_{\eta}\\
\text{we have that } |\log f'(t)-\log f'(s)| <B'_{\eta}|t-s|^{\alpha}
\end{matrix}
\Biggl\}.
\end{multline}
From Theorem~\ref{thrLocObsHolderLine} we also get that there exist $\delta''_{\eta}>0$ and $B''_{\eta}>0$ such that
\begin{multline}
\Biggl\{f \in \CZPlus(\R) \, :\quad
\begin{matrix}
\forall s\neq t \in \R \text{ with } |t-s|<  \delta'_{\eta} \text{ we have }\\
\text{that } |\log f'(t)-\log f'(s)| <B'_{\eta}|t-s|^{\alpha}
\end{matrix}\Biggl\}\\
\subset
\bigcap_{x\in [-\eta^{-1},\, \eta^{-1}]} 
\ObsHolLoc{\alpha} \! \left(x, \delta''_{\eta}, B''_\eta\right).
\end{multline}
Therefore,
\begin{equation}
\mathcal{Z}_{\sigma^2}^{-1} \cdot  
\left(\EmbTan{\beta}_{\sharp} \FMeas{\sigma^2}\right)_{*} \!
\Big(
\cap_{x\in [-\eta^{-1}, \eta^{-1}]} \ObsHolLoc{\alpha} \! \left(x, \delta''_{\eta}, B''_\eta \right) 
\Big) > 1- \eta,
\end{equation}
proving \eqref{eqPrpTightProof1}

\end{proof}

\subsection{Proof of Theorem~\ref{thrLocalLimit}}
\label{sectThrLocalProofFinal}

Since $\lim_{\sigma, \beta \to \infty} \sigma^2/\beta = \varkappa^2$, we have that measures $\mathcal{Z}^{-1}_{\sigma^2} \cdot \d\EmbTan{\beta}_{\sharp}\FMeas{\sigma^2}$ are tight by Proposition~\ref{prpTightnessLoc} and Corollary~\ref{crlCpmctFromObs}.
From correlation functions convergence, proved in Proposition~\ref{prpCorrConv}, we deduce that any limiting measure $\d\mathcal{P}$ has the same correlation functions of observables $\obs{f}{s_j}{t_j}_0$ as $\d\FMInf{\varkappa^2}$ (see Theorem~\ref{thrLineCorrelations}).
Finally, using that these correlation functions characterise the measure uniquely (see Theorem~\ref{thrUniqLineSchw}), we get that any limiting measure $\d\mathcal{P}$ coincides with $\d\FMInf{\varkappa^2}$.
Thus, we deduce that measures $\mathcal{Z}^{-1}_{\sigma^2} \cdot \d\EmbTan{\beta}_{\sharp}\FMeas{\sigma^2}$ weakly converge to $\d\FMInf{\varkappa^2}$.

\begin{proof}[Proof of Theorem~\ref{thrLocalLimit}]
To conclude the proof of the Theorem, it remains to show that measures $\mathcal{Z}^{-1}_{\sigma^2} \cdot \d\Emb{\beta}_{\sharp}\FMeas{\sigma^2}$ have the same limit as $\mathcal{Z}^{-1}_{\sigma^2} \cdot \d\EmbTan{\beta}_{\sharp}\FMeas{\sigma^2}$.
We demonstrate this by showing that for any open $U\subset C^1(\R)$,
\begin{equation}\label{eqProofLocalConv1}
\liminf_{\sigma, \beta \to\infty} 
\mathcal{Z}^{-1}_{\sigma^2} \cdot \Emb{\beta}_{\sharp}\FMeas{\sigma^2}(U)
\geq
\FMInf{\varkappa^2}(U).
\end{equation}

\medskip

\begin{lmm}
There exists a sequence of open $U_n \subset C^1(\R)$, such that $\cup_{n=1}^{\infty} U_n = U$, and for every $n$ we have that $\Clos(U_n) \subset U$ and $U_n \subset U_{n+1}$.
Here $\Clos$ is the topological closure.
\end{lmm}
\begin{proof}[Proof of the Lemma]
Since $C^1(\R)$ is a metric space, for every $f\in  U$ there exists an open $V_f\subset C^1(\R)$, such that $f\in V_f$ and $\Clos (V_f) \subset U$.
Using the fact that $C^1(\R)$ is second-countable, we can extract a countable subsequence $\{V_{f_k}\}_{k=1}^{\infty}$, such that $\cup_{k=1}^{\infty} V_{f_k} = U$.
Now we take $U_n = \cup_{k=1}^n V_{f_k}$.
\end{proof}

Inequality \eqref{eqProofLocalConv1} will follow, if we show that for all $n$,
\begin{equation}\label{eqProofLocalConv2}
\liminf_{\sigma, \beta \to\infty}
\mathcal{Z}^{-1}_{\sigma^2} \cdot \Emb{\beta}_{\sharp}\FMeas{\sigma^2}(U)
\geq
\liminf_{\sigma, \beta \to\infty}
\mathcal{Z}^{-1}_{\sigma^2} \cdot \EmbTan{\beta}_{\sharp}\FMeas{\sigma^2}(U_n).
\end{equation}
Indeed, from \eqref{eqProofLocalConv2} and convergence of $\mathcal{Z}^{-1}_{\sigma^2} \cdot \d\EmbTan{\beta}_{\sharp}\FMeas{\sigma^2}$ to $\d\FMInf{\varkappa^2}$ we can deduce that
\begin{equation}
\liminf_{\sigma, \beta \to\infty}
\mathcal{Z}^{-1}_{\sigma^2} \cdot \Emb{\beta}_{\sharp}\FMeas{\sigma^2}(U)
\geq
\liminf_{\sigma, \beta \to\infty} 
\mathcal{Z}^{-1}_{\sigma^2} \cdot \EmbTan{\beta}_{\sharp}\FMeas{\sigma^2}(U_n)
\geq
\FMInf{\varkappa^2}(U_n),
\end{equation}
where the right-hand side converges to $\FMInf{\varkappa^2}(U)$ as $n\to\infty$ because $\cup_{n=1}^{\infty} U_n = U$.

The following Lemma is the key observation to establish \eqref{eqProofLocalConv2}. 
\begin{lmm}\label{lmmProofThrLocalConv1}
Given an open $U$ and a compact $\cmpct \subset U$ we have that for any $\beta$ large enough the following implication holds for all $\phi\in \Diff^1(\T)/\SL(2, \R)$
\begin{equation}
 \EmbTan{\beta}(\phi) \in \cmpct
 \quad
 \implies
 \quad
 \Emb{\beta}(\phi) \in U.
\end{equation}
\end{lmm}
\begin{proof}[Proof of the Lemma]
For $N>1$, let
\begin{equation}
M_N = \max_{f \in \cmpct , \,  t \in [-N, N]} \Big(|f(t)| + |f'(t)|\Big),
\end{equation}
which is finite by compactness of $\cmpct$ and $[-N, N]$.
Then, using definition \eqref{eqDefFTan} of $\EmbTan{\beta}$, we get that for any $\beta > 10 (N+ M_N)$, and any $\phi\in \Diff^1(\T)/\SL(2, \R)$ such that $\EmbTan{\beta}(\phi)\in\cmpct$, we have
\begin{equation}\label{eqProofThrLocalProofLemma1}
\sup_{t\in [-N, N]} 
\left| 
\EmbTan{\beta}_{\phi}(t)-\Emb{\beta}_{\phi}(t)
\right|
\leq
\sup_{x\in [-M_N, M_N]} 
\left|
x-
\tfrac{\beta}{\pi} \arctan\left( \tfrac{\pi}{\beta} \, x\right) 
\right|,
\end{equation}
and
\begin{equation}\label{eqProofThrLocalProofLemma2}
\sup_{t\in [-N, N]} 
\left| 
\frac{\d}{\d t}\EmbTan{\beta}_{\phi}(t)-\frac{\d}{\d t}\Emb{\beta}_{\phi}(t)
\right|
\leq
\sup_{x\in [-M_N, M_N]} 
\left|
\frac{\d}{\d x}\left(
1-\tfrac{\beta}{\pi} \arctan\left( \tfrac{\pi}{\beta} \, x\right) 
\right)
\right|
\cdot M_N.
\end{equation} 
There exists $B(N)>0$ such that for any $\beta > B(N)$ we have that right-hand sides of both \eqref{eqProofThrLocalProofLemma1} and \eqref{eqProofThrLocalProofLemma2} are less than $1/N$.
In particular, we get that for any $\beta > B(N)$,
\begin{equation}
\distCR \Big(
\EmbTan{\beta}(\phi), \,
\Emb{\beta}(\phi)
\Big)
<\frac{10}{N}
\end{equation}
whenever $\EmbTan{\beta}(\phi)\in \cmpct$.
Thus, we can deduce that the assertion of the Lemma holds for any $\beta > B(N)$, if we take $N$ to be large enough (for example
\begin{equation}
N =  \frac{20}{\distCR \Big(\cmpct,\,  C^1(\R) \backslash U\Big)}
\end{equation}
will suffice).

\end{proof}

Now we prove \eqref{eqProofLocalConv2}.
By Proposition~\ref{prpTightnessLoc} and Corollary~\ref{crlCpmctFromObs} measures $\mathcal{Z}^{-1}_{\sigma^2} \cdot \d\EmbTan{\beta}_{\sharp}\FMeas{\sigma^2}$ are tight, meaning that there exists a sequence of compact $\cmpct_m\subset C^1(\R)$ such that 
\begin{equation}
\liminf_{\sigma, \beta\to\infty} \mathcal{Z}^{-1}_{\sigma^2} \cdot \EmbTan{\beta}_{\sharp}\FMeas{\sigma^2} (\cmpct_m) > 1-\frac{1}{m}.
\end{equation}

Therefore, applying Lemma~\ref{lmmProofThrLocalConv1} for $\cmpct = \Clos (U_n) \cap \cmpct_m$,
\begin{equation}
\liminf_{\sigma, \beta\to\infty}
\mathcal{Z}^{-1}_{\sigma^2} \cdot \Emb{\beta}_{\sharp}\FMeas{\sigma^2} (U )
\geq
\liminf_{\sigma, \beta\to\infty}
\mathcal{Z}^{-1}_{\sigma^2} \cdot \EmbTan{\beta}_{\sharp}\FMeas{\sigma^2} \Big(\Clos (U_n) \cap \cmpct_m\Big)
\geq
\liminf_{\sigma, \beta\to\infty}
\mathcal{Z}^{-1}_{\sigma^2} \cdot \EmbTan{\beta}_{\sharp}\FMeas{\sigma^2} (U_n)
- \frac{1}{m}.
\end{equation}
We obtain \eqref{eqProofLocalConv2} by taking $m\to \infty$.

\end{proof}

\section{Global structure proofs}
\label{sectThrGlobProof}

Denote
\begin{equation}
\cross{\phi}{t_1}{t_2}{t_3}{t_4}
=
\frac{\obs{\phi}{t_1}{t_3}\obs{\phi}{t_2}{t_4}}{\obs{\phi}{t_1}{t_2} \obs{\phi}{t_3}{t_4}}
=
\frac{\sin\Big(\pi\big[\phi(t_2)-\phi(t_1)\big]\Big) \sin\Big(\pi\big[\phi(t_4)-\phi(t_3)\big]\Big)}
{\sin\Big(\pi\big[\phi(t_3)-\phi(t_1)\big]\Big) \sin\Big(\pi\big[\phi(t_4)-\phi(t_2)\big]\Big)}.
\end{equation} 
These cross-ratios will play a key role in the proofs of Theorem~\ref{thrGlobalConvJumpBound} and Theorem~\ref{thrGlobalConvJumpPossible}.
They are convenient because they allow us to express jumps in $\phi$ via observables in a $\SL(2, \R)$ invariant manner, as stated in the following Proposition.

\begin{prp}\label{prpGlobalProofObsJump}
Let $s_1, s_2, s_3, s_4, s_5=s_1 \in \T$ be distinct points which lie on $\T$ in the clockwise order.
Then 
\begin{enumerate}
\item For any $\phi\in \Diff^1(\T)$ we have
\begin{equation}\label{eqPrpGlobalProofEpsDelta1}
\cross{\phi}{s_1}{s_2}{s_3}{s_4} \in (0, 1).
\end{equation}

\item
For any $\eps>0$ there exists $\delta>0$ such that for any $\phi\in \Diff^1(\T)$ satisfying
\begin{equation} \label{eqPrpGlobalProofEpsDelta2}
\forall i\in \{1, 2, 3, 4\}: \qquad \phi(s_{i+1}) - \phi(s_i) > \eps,
\end{equation}
we have 
\begin{equation} \label{eqPrpGlobalProofEpsDelta3}
\cross{\phi}{s_1}{s_2}{s_3}{s_4} \in (\delta, 1-\delta).
\end{equation}

\end{enumerate}

\end{prp}

\begin{proof}
By rotating the argument, we can assume that $s_1=0$.
Also, let $\eps>0$ be such that \eqref{eqPrpGlobalProofEpsDelta2} holds.

Denote $f_j = -\cot\big(\frac{\eps}{2} + \pi \,\phi (s_j) \big)$.
Then it is easy to check that 
\begin{equation}
\cross{\phi}{s_1}{s_2}{s_3}{s_4}
=
\frac{(f_2-f_1)(f_4-f_3)}{(f_3-f_1)(f_4-f_2)}.
\end{equation}
It is obvious that
\begin{equation}
0 < \frac{(f_2-f_1)(f_4-f_3)}{(f_3-f_1)(f_4-f_2)} < 1,
\end{equation}
which proves~\eqref{eqPrpGlobalProofEpsDelta1}.

Now we prove~\eqref{eqPrpGlobalProofEpsDelta3}.
Note that
\begin{equation}
\frac{(f_2-f_1)(f_4-f_3)}{(f_3-f_1)(f_4-f_2)}
=
\left(
1+ \frac{(f_3-f_2)}{(f_2-f_1)}
\right)^{-1}
\times
\,\,
\left(
1+ \frac{(f_3-f_2)}{(f_4-f_3)}
\right)^{-1}.
\end{equation}
From \eqref{eqPrpGlobalProofEpsDelta2} we have
$f_1, f_2, f_3, f_4
\in
(-\cot \frac{\eps}{2}, \cot \frac{\eps}{2} )$.
Since $\forall t\in (0, 1-\eps): \, |\cot (t + \eps) - \cot t|\geq \eps$,
\begin{equation}
\left(1+\frac{2 \cot \frac{\eps}{2} }{\eps}\right)^{-2}
\leq 
\left(
1+ \frac{(f_3-f_2)}{(f_2-f_1)}
\right)^{-1}
\!\!
\times
\left(
1+ \frac{(f_3-f_2)}{(f_4-f_3)}
\right)^{-1}
\leq 
\left(1+\frac{\eps}{2 \cot \frac{\eps}{2} }\right)^{-2}.
\end{equation}
\end{proof}

Now our goal is to show that in the $\sigma\to\infty$ limit we have that cross-ratios $\cross{\phi}{s_1}{s_2}{s_3}{s_4}$ become close either to $0$ or to $1$ with high probability with respect to $\mathcal{Z}_{\sigma^2}^{-1}\cdot \d\FMeas{\sigma^2}$, which we do in Lemma~\ref{lmmGlobalLimitObs}.
There we also show that these cross-ratios become close to $0$ if $s_3$ and $s_4$ are close.

\subsection{Schwarzian Measure on the interval}
We start by relating expectations with respect to $\d\FMeas{\sigma^2}$ to expectations with respect to Brownian Bridges $\d \WS{\sigma^2}{a}{T}$.
Let
\begin{equation}
h_T(a, \sigma) 
= \int_{\R_+} \cos (a\, k) \, \exp\left( - \frac{(1-T)\sigma^2}{2}\cdot k^2 \right)
\sinh(2\pi k)\, 2k \, \d k.
\end{equation}

\begin{lmm}\label{lmmGlobalUniqFuncEqual}
Let $\Big\{\obs{f}{s_{j}}{t_j}\Big\}_{j=1}^{N}$ be a set of observables (not necessarily non-interlaced) on $[0, T]$ for some $T\in (0, 1)$.
Then we have that for any continuous bounded function $F\in C(\R^N_+)$,
\begin{multline}\label{eqLmmGlobalUniqFuncEqual}
\int_{\R} 
\left[
\int_{C[0, T]}
 F\left(\Big\{\obs{\B_{\xi}}{s_j}{t_j}_0
 \Big\}_{j=1}^{N}
 \right)
\d \WS{\sigma^2}{a}{T}(\xi) 
\right]
h_T(a, \sigma) \d a\\
=
\int_{\Diff^1(\T)/\SL(2, \R)}
 F\left(\Big\{\obs{\phi}{s_{j}}{t_j}\Big\}_{j=1}^{N}\right) \d\FMeas{\sigma^2}(\phi).
\end{multline}
\end{lmm}
\begin{proof}
 
First, we observe that \eqref{eqLmmGlobalUniqFuncEqual} holds when observables are non-interlaced and $F$ is a monomial  (i.e. for non-interlaced correlation functions).
Indeed, this follows from explicit formulae for correlation functions of $\d \WS{\sigma^2}{a}{T}$ and $\d\FMeas{\sigma^2}$, which were derived in \citep{LosevCorr} (we recall them in Proposition~\ref{prpIntervalExpObsFormula_corr} and Theorem~\ref{thrMainCorrelations_corr}),
and the fact that with our normalisation of the Fourier transform, 
\begin{equation}
\int_{\R}\left(\int_{\R_+} 
\frac{\cos(a \, k_1)}{\pi} 
G(k_1) \d k_1\right)
\cdot
\left(\int_{\R_+} 
 \cos(a \, k_2)
 H(k_2)\d k_2\right)
\d a
=
\int_{\R_+} G(k) H(k) \d k.
\end{equation}

Secondly, we notice the existence of exponential moments for both sides.
For the right-hand side this is precisely Proposition~\ref{prpExpMomentCirc_corr}.
For the left-hand side it follows from Proposition~\ref{prpObsExpMomentBound_corr} that for any $s\neq t \in \T$,
\begin{multline}
\int_{\R} 
\left[
\int_{C[0, T]}
 \exp\left(\frac{8}{\sigma^2}{\obs{\B_{\xi}}{s}{t}_0}
 \right)
\d \WS{\sigma^2}{a}{T}(\xi) 
\right]
\Big| h_T(a, \sigma) \Big| 
\d a\\
\leq 
\int_{\R} 
\frac{1}{\sqrt{2\pi T}\sigma} \,
\exp\left(\frac{2000}{(t-s)\sigma^2}-\frac{a^2}{2 T\sigma^2}\right)\,
\Big| h_T(0, \sigma) \Big| \d a,
\end{multline}
which is finite.
Combining this with \eqref{eqLmmGlobalUniqFuncEqual} for monomials, we get that \eqref{eqLmmGlobalUniqFuncEqual} holds for imaginary exponentials (i.e. Fourier transform) of non-interlaced observables. 
Thus, we can deduce that the lemma holds whenever $\Big\{\obs{\phi}{s_{j}}{t_j}\Big\}_{j=1}^{N}$ are non-interlaced. 

Finally, we are going to show that any continuous bounded function $F\in C(\R^N_+)$ of some set of observables can be rewritten as a  continuous bounded function $G$ of some non-interlaced set of observables.
Let $0\leq x_1 < x_2 < \ldots < x_M <1$ be all of the points $\{s_j\}_{j=1}^N$ and $\{t_j\}_{j=1}^N$ sorted in the increasing order.
Consider the non-interlaced set of $2M-3$ observables 
\begin{equation}
\Big\{\obs{\phi}{x_{j}}{x_{j+1}}\Big\}_{j=1}^{M-1} 
\cup
\Big\{\obs{\phi}{x_{1}}{x_{j}}\Big\}_{j=3}^{M},
\end{equation}
which we also denote by $\{\mathcal{O}(\phi;\, k)\}_{k=1}^{2M-3}$ as a shorthand.
Similarly, we denote non-interlaced observables 
\begin{equation}
\Big\{\obs{\B_{\xi}}{x_{j}}{x_{j+1}}_0\Big\}_{j=1}^{M-1} 
\cup
\Big\{\obs{\B_{\xi}}{x_{1}}{x_{j}}_0\Big\}_{j=3}^{M}
\end{equation}
by $\{\mathcal{O}_0(\B_{\xi};\, k)\}_{k=1}^{2M-3}$.
We claim that for any $1\leq i < j \leq M$ there exists a continuous function $H_{i, j}: \R_+^{2M-3} \to \R_+$ such that 
\begin{align}
\obs{\phi}{x_i}{x_j} = H_{i, j}\Big( \{\mathcal{O}(\phi;\, k)\}_{k=1}^{2M-3}\Big),
\\
\obs{\B_{\xi}}{x_i}{x_j}_0 = H_{i, j}\Big( \{\mathcal{O}_0(\B_{\xi}; \, k)\}_{k=1}^{2M-3}\Big).
\end{align}
This is easy to prove by induction on $|j-i|$ using Lemma~\ref{lmmObsAlgEquality} and the fact that the exact same equality (as in Lemma~\ref{lmmObsAlgEquality}) holds for $\obs{\phi}{\cdot}{\cdot}$ in place of $\obs{f}{\cdot}{\cdot}_0$, namely that
\begin{equation}\label{eqLmmGlobalProofBBReduction1}
\frac{1}{\obs{\phi}{\tau_1}{\tau_3} \obs{\phi}{\tau_2}{\tau_4}} = 
\frac{1}{\obs{\phi}{\tau_1}{\tau_2} \obs{\phi}{\tau_3}{\tau_4}}
+\frac{1}{\obs{\phi}{\tau_1}{\tau_4} \obs{\phi}{\tau_2}{\tau_3}}.
\end{equation}
The latter was proved in \citep[Lemma~36]{LosevLDP} (alternatively, it also follows from Lemma~\ref{lmmObsAlgEquality} after substituting $f = \tan(\pi \phi)$).
When proving the existence of $H_{i, j}$ by induction on $|j-i|$ we use Lemma~\ref{lmmObsAlgEquality} and \eqref{eqLmmGlobalProofBBReduction1} with $\tau_1 = x_1$, $\tau_2=x_i$, $\tau_3=x_{i+1}$, $\tau_4=x_{j}$.

Therefore, we deduce that there exists a bounded continuous $G\in C(\R^{2M-3}_+)$
\begin{equation}
F\left(\Big\{\obs{\phi}{s_{j}}{t_j}\Big\}_{j=1}^{N}\right)
=
G\Big( \{\mathcal{O}(\phi;\, k)\}_{k=1}^{2M-3}\Big)
\end{equation}
and
\begin{equation}
F\left(\Big\{\obs{\B_{\xi}}{s_{j}}{t_j}_0\Big\}_{j=1}^{N}\right)
=
G\Big( \{\mathcal{O}_0(\B_{\xi};\, k)\}_{k=1}^{2M-3}\Big).
\end{equation}
Since we have already proved the Lemma for non-interlaced observables, the desired result follows immediately.

\end{proof}

Denote
\begin{equation}
\cross{f}{s_1}{s_2}{s_3}{s_4}_0
=
\frac{\obs{f}{s_1}{s_3}_0 \obs{f}{s_2}{s_4}_0}
{\obs{f}{s_1}{s_2}_0 \obs{f}{s_3}{s_4}_0}
=
\frac{\left| f(s_2)-f(s_1)\right| \cdot \left|f(s_4)-f(s_3)\right|}
{\left|f(s_3)-f(s_1)\right| \cdot \left|f(s_4)-f(s_2)\right|}.
\end{equation}
Now we calculate the law of $\cross{\B_{\xi}}{s_1}{s_2}{s_3}{s_4}_0$ in the large temperature limit of the Brownian Bridges.

\begin{lmm}\label{lmmCrossBBConv}
Let $0\leq s_1 < s_2 < s_3 < s_4 \leq T$.
Then for any bounded continuous function $F\in C(\R_+)$ and any $a\in \R$, 
\begin{equation}
\lim_{\sigma\to\infty}
\left[
\frac{1}
{\WS{\sigma^2}{a\sigma}{T}\Big(C[0, T]\Big)}
\cdot
\int_{C[0, T]}
F\left(
 \cross{\B_{\xi}}{s_1}{s_2}{s_3}{s_4}_0
 \right)
\d \WS{\sigma^2}{a\sigma}{T}(\xi) 
\right]
= q \, F(1) + (1-q) \, F(0),
\end{equation}
where
\begin{equation}
q 
=
\WS{1}{a}{T}\left\{
\xi\in C[0, T]: \,
\max_{\tau\in [s_1, s_{2}]}\xi(\tau)
>
\max_{\tau\in [s_2, s_{3}]}\xi(\tau)
\, \text{ and } \,
\max_{\tau\in [s_2, s_{3}]}\xi(\tau)
<
\max_{\tau\in [s_3, s_{4}]}\xi(\tau)
\right\}.
\end{equation}

\end{lmm}
\begin{proof}
Denote
\begin{equation}
\B_{\xi}^{\sigma^2}(t) \coloneqq 
\int_0^t \exp\left\{\sigma \xi(s)
\right\}\d s,
\end{equation}
so that
\begin{multline}
\frac{1}
{\WS{\sigma^2}{a\sigma}{T}\Big(C[0, T]\Big)}
\cdot
\int_{C[0, T]}
F\left(
 \cross{\B_{\xi}}{s_1}{s_2}{s_3}{s_4}_0
 \right)
\d \WS{\sigma^2}{a\sigma}{T}(\xi) \\
=
\frac{1}
{\WS{1}{a}{T}\Big(C[0, T]\Big)}
\cdot
\int_{C[0, T]}
F\left(
 \cross{\B_{\xi}^{\sigma^2}}{s_1}{s_2}{s_3}{s_4}_0
 \right)
\d \WS{1}{a}{T}(\xi) .
\end{multline}
Observe that
\begin{equation}\label{eqLmmProofCrossConv1}
\cross{\B_{\xi}^{\sigma^2}}{s_1}{s_2}{s_3}{s_4}_0
=\left(
1+ \frac{|\B_{\xi}^{\sigma^2}(s_3) - \B_{\xi}^{\sigma^2}(s_2)|}{|\B_{\xi}^{\sigma^2}(s_2) - \B_{\xi}^{\sigma^2}(s_1)|}
\right)^{-1}
\times \, \,
\left(
1+ \frac{|\B_{\xi}^{\sigma^2}(s_3) - \B_{\xi}^{\sigma^2}(s_2)|}{|\B_{\xi}^{\sigma^2}(s_4) - \B_{\xi}^{\sigma^2}(s_3)|}
\right)^{-1}
.
\end{equation}
Also, for any $s< t$ we have almost surely for $\xi\sim \d\WS{1}{a}{T}$ that 
\begin{equation}\label{eqLmmProofCrossConv2}
\lim_{\sigma\to\infty} \frac{\log \left| \B_{\xi}^{\sigma^2}(t)-\B_{\xi}^{\sigma^2}(s)\right| }{\sigma} = \max_{\tau\in [s, t]}\xi(\tau).
\end{equation}
For $j\in \{1, 2, 3\}$ we let $m_j(\xi) = \max_{\tau\in [s_j, s_{j+1}]}\xi(\tau)$.
Almost surely all $\{m_j(\xi)\}_{j=1}^3$ are distinct.
Therefore, 
\begin{equation}
\lim_{\sigma\to \infty}
\cross{\B_{\xi}^{\sigma^2}}{s_1}{s_2}{s_3}{s_4}_0
=
\begin{cases}
1, \qquad & \text{if } \, m_1(\xi) > m_2(\xi)\, \text{ and } \, m_2(\xi) < m_3(\xi) \\
0, \qquad & \text{if } \, m_1(\xi) < m_2(\xi)\, \text{ or } \, m_2(\xi) > m_3(\xi)
\end{cases}.
\end{equation}
Now the Lemma follows from the Dominated Convergence Theorem.

\end{proof}

\subsection{Proofs of Theorem~\ref{thrGlobalConvJumpBound} and Theorem~\ref{thrGlobalConvJumpPossible}}

In Proposition~\ref{prpGlobalProofObsJump} we showed that the cross-ratios $\cross{\phi}{s_1}{s_2}{s_3}{s_4}$ always take values in $(0, 1)$, and that $\phi$ having $4$ jumps translates to this observable being bounded away from both $0$ and $1$. 
The main ingredient in the proofs of Theorems~\ref{thrGlobalConvJumpBound} and~\ref{thrGlobalConvJumpPossible} is that we can now show that these cross-ratios are getting close either to $0$ or to $1$ in the $\sigma\to\infty$ limit.
We also show that these observables are close to $0$ if $s_3$ and $s_4$ are close.

\begin{lmm}\label{lmmGlobalLimitObs}
Let $s_1, s_2, s_3, s_4 \in \T$ be distinct points which lie on $\T$ in the clockwise order.
Then
\begin{equation}
\lim_{\sigma\to \infty}
\left[
\mathcal{Z}_{\sigma^2}^{-1}
\int_{\Diff^1(\T)/\SL(2, \R)} 
\cross{\phi}{s_1}{s_2}{s_3}{s_4}\left(1- \cross{\phi}{s_1}{s_2}{s_3}{s_4}\right)
\d\FMeas{\sigma^2} (\phi)\right]
=
0.
\end{equation}
Suppose now that $s_1, s_2, t$ are distinct points which lie on $\T$ in the clockwise order.
Then
\begin{equation}
\lim_{
\substack{
s_4\to t+\\
s_3\to t-
}
}
\left[
\limsup_{\sigma\to \infty}
\left[
\mathcal{Z}_{\sigma^2}^{-1}
\int_{\Diff^1(\T)/\SL(2, \R)} 
\cross{\phi}{s_1}{s_2}{s_3}{s_4}
\d\FMeas{\sigma^2} (\phi)\right]
\right]
=
0.
\end{equation}
\end{lmm}
\begin{proof}
By rotating the arguments we can assume that $s_1=0$. 
Fix some $T\in (s_4, 1)$.
Denote
\begin{equation}
P_1(x) = 
\begin{cases}
x(1-x)\qquad & \text{if } x\in [0, 1];\\
0 \qquad & \text{if } x\notin [0, 1];
\end{cases}
\qquad
\text{and}
\qquad
P_2(x) = 
\begin{cases}
x\qquad & \text{if } x\in [0, 1];\\
0 \qquad & \text{if } x < 0; \\
1 \qquad & \text{if } x > 1.
\end{cases}
\end{equation}

Then from Lemma~\ref{lmmGlobalUniqFuncEqual} we have that for $j\in \{1, 2\}$,
\begin{multline}\label{eqGlobalProofPolyConv0}
\limsup_{\sigma\to \infty}
\left[
\mathcal{Z}_{\sigma^2}^{-1}
\int_{\Diff^1(\T)/\SL(2, \R)} 
P_j\Big(\cross{\phi}{s_1}{s_2}{s_3}{s_4}\Big)
\d\FMeas{\sigma^2} (\phi)
\right]
\\
=
\limsup_{\sigma\to \infty}
\int_{\R}
\mathcal{Z}_{\sigma^2}^{-1} \,
g_j(a, \sigma) \,
\WS{\sigma^2}{a\sigma}{T}\Big(C[0, T]\Big) \,
 h_T(a\sigma, \sigma) \sigma \d a
 ,
\end{multline}
where
\begin{equation}
g_j(a, \sigma)
=
\frac{1}
{\WS{\sigma^2}{a\sigma}{T}\Big(C[0, T]\Big)}
\cdot
\int_{C[0, T]}
 P_j\Big(\cross{\B_{\xi}}{s_1}{s_2}{s_3}{s_4}_0\Big)
\d \WS{\sigma^2}{a\sigma}{T}(\xi).
\end{equation}

We have that
\begin{equation}
\WS{\sigma^2}{a\sigma}{T}\Big(C[0, T]\Big)
=
\frac{1}{\sqrt{2\pi T} \, \sigma} \,
\exp\left(-\frac{a^2}{2T}\right),
\end{equation}
and from Proposition~\ref{prpMainPartFunct}, it is easy to see that
\begin{equation}
| h_T(a\sigma, \sigma)|
\leq
\mathcal{Z}\Big((1-T) \sigma^2\Big).
\end{equation}
Thus, 
we deduce that for some constant $C(T)>0$ and any $\sigma>1$,
\begin{equation}
\left|
\mathcal{Z}_{\sigma^2}^{-1} \,
\WS{\sigma^2}{a\sigma}{T}\Big(C[0, T]\Big) \,
 h_T(a\sigma, \sigma) \sigma
 \right|
 \leq 
 \left|
 \frac{1}{\sqrt{2\pi T}} \,
\exp\left(-\frac{a^2}{2T}\right)\,
\frac{\mathcal{Z}\big((1-T) \sigma^2\big)}{\mathcal{Z}(\sigma^2)}
 \right|
 \leq
 C(T)
\exp\left(-\frac{a^2}{2T}\right),
\end{equation}
where for the last inequality we used the explicit formula for the partition function $\mathcal{Z}$ (see Proposition~\ref{prpMainPartFunct}).

Therefore, 
\begin{equation}
\limsup_{\sigma\to \infty}
\left|
\int_{\R}
\mathcal{Z}_{\sigma^2}^{-1} \,
g_j(a, \sigma) \,
\WS{\sigma^2}{a\sigma}{T}\Big(C[0, T]\Big) \,
 h_T(a\sigma, \sigma) \sigma \d a
 \right|
 \leq
 C(T)
\limsup_{\sigma\to \infty}
 \int_{\R} g_j(a, \sigma) \exp\left(-\frac{a^2}{2T}\right) \d a.
\end{equation}
Now the desired result follows from Lemma~\ref{lmmCrossBBConv} and the observation that
\begin{equation}
\lim_{\substack{
s_4\to t+\\
s_3\to t-
}}
\WS{1}{a}{T}\left\{
\xi\in C[0, T]: \,
\max_{\tau\in [s_1, s_{2}]}\xi(\tau)
>
\max_{\tau\in [s_2, s_{3}]}\xi(\tau)
\, \text{ and } \,
\max_{\tau\in [s_2, s_{3}]}\xi(\tau)
<
\max_{\tau\in [s_3, s_{4}]}\xi(\tau)
\right\} = 0.
\end{equation}

\end{proof}

\begin{proof}[Proof of Theorem~\ref{thrGlobalConvJumpBound}]
Suppose that $\phi$ is such that 
\begin{equation}
\left|
\left\{
j\in\{1, 2, \ldots N \}: \,
\phi(t_{j+1}) - \phi(t_{j}) > \eps
\right\}
\right| \geq 4.
\end{equation}
Then there exist $s_1, s_{2}, s_{3}, s_{4} \in \{t_j\}_{j=1}^N$ such that $0 \leq s_1<s_2<s_3<s_4<1$ and
\begin{equation}
\forall j\in\{1, 2, 3, 4\}: \qquad \phi(s_{j+1}) - \phi(s_{j}) > \eps ,
\end{equation}
where we use the convention $s_5=s_1\in\T$.
Thus, by Proposition~\ref{prpGlobalProofObsJump}, there exists $\delta>0$ such that for any $\phi$ as above
\begin{equation}
\cross{\phi}{s_1}{s_2}{s_3}{s_4} \in (\delta, 1-\delta).
\end{equation}
Therefore, we can take
\begin{equation}
\eventGlobal_{N, \eps}
=
\bigcap_{\{s_1, s_{2}, s_{3}, s_{4}\} \subset \{t_j\}_{j=1}^N}
\left\{
[\phi] \in \Diff^1(\T)/\SL(2, \R)
:\,
\cross{\phi}{s_1}{s_2}{s_3}{s_4} \notin (\delta, 1-\delta)
\right\},
\end{equation}
where the intersection is taken over quadruples of pairwise distinct $\{s_1, s_{2}, s_{3}, s_{4}\}$.
From Lemma~\ref{lmmGlobalLimitObs} we deduce that for any such $\{s_1, s_{2}, s_{3}, s_{4}\}$,
\begin{equation}
\lim_{\sigma\to \infty }\mathcal{Z}^{-1}_{\sigma^2} \cdot \FMeas{\sigma^2} 
\left\{
[\phi] \in \Diff^1(\T)/\SL(2, \R)
:\,
\cross{\phi}{s_1}{s_2}{s_3}{s_4} \in (\delta, 1-\delta)
\right\} = 0,
\end{equation}
which finishes the proof.

\end{proof}

\begin{proof}[Proof of Theorem~\ref{thrGlobalConvJumpPossible}]

Without loss of generality we can assume that $\eps$ is so small that $\eps<10^{-3}\eta$.
In this proof we will fix the gauge.
For a conjugacy class $[\phi] \in \Diff^1(\T)/\SL(2, \R)$ we always assume that $\phi\in \Diff^1(\T)$ is normalised so that
\begin{equation}
\phi(0) = 0,
\qquad
\phi(1/3) = 1/3,
\qquad
\phi(2/3) = 2/3.
\end{equation}

We prove the Theorem by considering all $3$ cases for $k=1, 2, 3$ separately.
The most important case is $k=3$, and the other two cases will be obtained just by changing the gauge.

\medskip
\textbf{Case $k=3$.}

Consider the following events for $j\in \{0, 1, 2\}$,
\begin{align}
E_{2j+1} &= 
\left\{
[\phi] \in \Diff^1(\T)/\SL(2, \R)
:\,\cross{\phi}{\tfrac{j-2}{3}}{\tfrac{j-1}{3}}{\tfrac{j}{3}-\eta}{\tfrac{j}{3}}
< 10^{-10}
\right\};\\
E_{2j+2} &= 
\left\{
[\phi] \in \Diff^1(\T)/\SL(2, \R)
:\,\cross{\phi}{\tfrac{j-2}{3}}{\tfrac{j-1}{3}}{\tfrac{j}{3}}{\tfrac{j}{3}+\eta}
< 10^{-10}
\right\}.
\end{align}

From Lemma~\ref{lmmGlobalLimitObs} we get that there exists $\rho(\eta)>0$ such that 
\begin{align}
\liminf_{\sigma\to\infty}
\mathcal{Z}^{-1}_{\sigma^2}\cdot
\FMeas{\sigma^2} 
\left\{
\bigcap_{i\in\{1, \ldots, 6\}}E_i
\right\}
>
1-\rho(\eta)
\end{align}
and $\lim_{\eta\to 0} \rho(\eta) =0$.

Note that, because of the chosen gauge,
the following implications hold for all $j\in\{0, 1, 2\}$,
\begin{align}
[\phi] \in E_{2j+1}\qquad &\implies \qquad  \phi(\tfrac{j}{3}) - \phi(\tfrac{j}{3}-\eta)   < 10^{-5}; \\
[\phi] \in E_{2j+2}\qquad &\implies \qquad  \phi(\tfrac{j}{3}+\eta) - \phi(\tfrac{j}{3})   < 10^{-5}.
\end{align}

Now we partition $\T$ into small intervals.
Let $0=t_1<t_2<\ldots<t_N<t_{N+1} = 1$ be such that $\forall j\in\{1, \ldots, N\}:\, t_{j+1}-t_j <\eps$.
Then, by Theorem~\ref{thrGlobalConvJumpBound}, with large probability there exist no more than $3$ jumps of size $\eps/(100N)$.
In other words, there exists $E_0\subset \Diff^1(\T)/\SL(2, \R)$ such that
\begin{equation}
 E_0
\subset
\left\{
[\phi]\in \Diff^1(\T)/\SL(2, \R) : \,
\left|
\left\{
j\in\{ 1, 2, \ldots N \}: \,
\phi(t_{j+1}) - \phi(t_{j}) > \frac{\eps}{100N}
\right\}
\right| \leq 3 
\right\}
\end{equation} 
and
\begin{equation}
\lim_{\sigma\to\infty}
\mathcal{Z}^{-1}_{\sigma^2}\cdot
\FMeas{\sigma^2} 
\left(
E_0
\right) = 1.
\end{equation}
Note that because of the way we fixed the gauge,
we get that these $3$ jumps have to be of size $\approx 1/3$ and occur on $3$ intervals between $0, 1/3, 2/3$.
In other words, there exist $j_1, j_2, j_3 \in \{1, 2, \ldots N \}$ such that
\begin{equation}
t_{j_1}\in (0, 1/3),
\qquad
t_{j_2}\in (1/3, 2/3),
\qquad
t_{j_3}\in (2/3, 1),
\end{equation} 
and also
\begin{equation}\label{eqThrGlobExampleProof1}
\phi(t_{j_1+1}) - \phi(t_{j_1})  > \frac{1}{3}-\eps,
\qquad
\phi(t_{j_2+1}) - \phi(t_{j_2})  > \frac{1}{3}-\eps,
\qquad
\phi(t_{j_3+1}) - \phi(t_{j_3}) > \frac{1}{3}-\eps.
\end{equation}
Moreover, if in addition $[\phi] \in \cap_{i\in\{1, \ldots, 6\}}E_i$, then $t_{j_1}, t_{j_2}, t_{j_3}$ also will be separated from $0, 1/3, 2/3$ by at least $2\eta/3$,
\begin{equation}\label{eqThrGlobExampleProof2}
t_{j_1}\in \left(\frac{2\eta}{3}, \, \frac{1}{3}-\frac{2\eta}{3}\right),
\qquad
t_{j_2}\in \left(\frac{1}{3}+\frac{2\eta}{3}, \, \frac{2}{3}-\frac{2\eta}{3}\right),
\qquad
t_{j_3}\in \left(\frac{2}{3}+\frac{2\eta}{3}, \, 1-\frac{2\eta}{3}\right).
\end{equation} 
Therefore, taking $\evGlob_{3, \eta, \eps} = \cap_{i\in \{0, 1,\ldots, 6\}} E_i$ finishes the proof of the Theorem for $k=3$.

\medskip

\textbf{Case $k=2$.}
Consider the M\"{o}bius transformation $\psi \in \SL(2, \R)$ such that
\begin{equation}
\psi(0)=0;
\qquad
\psi(1/3) = 1/2;
\qquad
\psi(2/3) = 1-\eps.
\end{equation}
Denote $M = M(\eps)= \max_{t\in \T}\{\psi'(t), 1/\psi'(t)\}$.

Note that if $[\phi]\in \evGlob_{3, \eta, \eps/M}$, then, as shown above, there exist $t_{j_1}, t_{j_2}, t_{j_3}$ such that both \eqref{eqThrGlobExampleProof1} and \eqref{eqThrGlobExampleProof2} hold.
Then it is easy to see that 
\begin{equation}
\psi\circ \phi(t_{j_1+1}) - \psi\circ \phi(t_{j_1}) > \frac{1}{2}-\eps,
\qquad
\psi\circ\phi(t_{j_2+1}) - \psi\circ\phi(t_{j_2}) > \frac{1}{2}-\eps.
\end{equation}
Therefore, if we take $\evGlob_{2, \eta, \eps} = \evGlob_{3, \eta, \eps/M}$, then for any $[\phi]\in \evGlob_{2, \eta, \eps}$ we have that the representative $\psi\circ\phi$ satisfies the conditions desired in the statement of the Theorem.

\textbf{Case $k=1$.} 
The proof is similar to $k=2$, if we take the M\"{o}bius transformation $\psi \in \SL(2, \R)$ such that
\begin{equation}
\psi(0)=0;
\qquad
\psi(1/3) = 1-\eps;
\qquad
\psi(2/3) = 1-\frac{\eps}{2}.
\end{equation}
\end{proof}

\appendix
\section{Appendix}
\label{sect_Appendix}

\begin{stm}[\citep{LosevLDP}] \label{stmArccoshTaylor_LDP}
There exists $C>0$ such that for all $x\geq 0$ and $y\in [-\frac{1}{2}, \frac{1}{2}]$, 
\begin{equation}
\arccosh^2\left[\cosh\left(x\right)+y\right]
\geq
x^2 + 2y - C|y|(x^2+|y|),
\end{equation}
and for all $x\geq 10$ and $y\in [-\frac{1}{2}, \frac{1}{2}]$, 
\begin{equation}\label{eq_Arccosh_Taylor_Two_LDP}
\left|\arccosh^2\left[\cosh\left(x\right)+y\right]-x^2\right| 
\leq
C|y| \, e^{-x/2}.
\end{equation}
\end{stm}

\begin{proof}
Notice that for $u\geq 1$,
\begin{equation}
\frac{\d}{\d u}\arccosh^2\left[u\right]
=
2 \frac{\arccosh\left[u\right]}{\sqrt{u^2-1}},
\end{equation}
where both sides are analytic in $\D_1$.
Inequality \eqref{eq_Arccosh_Taylor_Two_LDP} follows since the right-hand side is smaller than $1/\sqrt{u}$ for large $u$.
Moreover, for some $C_1>0$,
\begin{equation}
\left|\frac{\d^2}{\d u^2}\arccosh^2\left[u\right]\right|
\leq  C_1,
\end{equation}
and
\begin{equation}
\frac{\d}{\d u}\arccosh^2\left[u\right]\Big|_{u=1} = 2.
\end{equation}
It is also easy to see that for some $C_2>0$,
\begin{equation}
\left|\frac{\d}{\d u}\arccosh^2\left[u\right] -2\right| 
\leq C_2\, \min\big\{ |u|-1, 1\big\}.
\end{equation}
Taylor expanding $\arccosh^2\left[\cosh\left(x\right)+y\right]$ in $y$ and using the fact that $\min\big\{ \cosh(x)-1, 1\big\}\leq C_3 x^2$ for some $C_3>0$ gives the desired result.
\end{proof}

\begin{prp}[\citep{LosevLDP}]\label{prpAppExpEst_LDP}
There exists $C>0$ such that for any $\sigma>0$, any $s, t\in \T$ such that $t\neq s\in \T$,
and any $\lambda\in [-C^{-1}, C^{-1}]$, we have
\begin{multline}
\int \exp\left\{\frac{2\lambda}{\sigma^2} \left(\obs{\phi}{0}{t} -\frac{1}{t}\right)\right\}\d \FMeas{\sigma^2}(\phi) \\
\leq
\frac{\sqrt{2}}{\sqrt{\pi t}\sigma}
\int_{\R_+^2}
\exp\left\{ -\frac{\alpha^2}{2t\sigma^2}
+ \frac{C_1|\lambda|(\alpha^2+ |\lambda|)}{2t\sigma^2}
\right\}
\exp\left(-\frac{(1-t)\sigma^2}{2}\cdot k_2^2\right)
\sinh(2\pi k_2)\, 2k_2 \d k_2
\d \alpha.
\end{multline}
\end{prp}
\begin{proof}
Using Proposition~\ref{prpMainPartFunct} and Theorem~\ref{thrMainCorrelations_corr} we can express exponential moments of $\obs{\phi}{0}{t}$, as
\begin{multline}
\int \exp\left\{\frac{2 \lambda}{\sigma^2}\, \obs{\phi}{0}{t} \right\}\d \FMeas{\sigma^2}(\phi) 
 = \int_{\R_+} \exp\left(-\frac{\sigma^2}{2}\cdot k_2^2\right)\sinh(2\pi k_2)\, 2k_2\d k_2 \\
 +
\sum_{l=1}^{\infty}\int_{\R_+^{2}} 
\frac{\Gamma\big(\frac{l}{2} \pm i k_1 \pm i k_2\big)}{2\pi^2\, \Gamma(l)}\cdot\frac{\lambda^l}{l!}\cdot
\exp\left(-\frac{t\sigma^2}{2}\cdot k_1^2 -\frac{\big(1-t\big)\sigma^2}{2}\cdot k_2^2\right)  
\\
\times
\sinh(2\pi k_1)\, 2 k_1 \sinh(2\pi k_2)\, 2 k_2 \d k_1 \d k_2,
\end{multline}
for $\lambda\in [-1, 1]$. Here the right-hand side converges absolutely, because it is dominated by the same expression with $\lambda=1$, which converges because all terms are positive and exponential moments of $\obs{\phi}{0}{t}$ exist (see Proposition~\ref{prpExpMomentCirc_corr}).

First, we give an upper bound for the integral over $k_1$.
It was proved in \citep[Appendix]{LosevCorr} that
\begin{multline}\label{eqPrpKeyArccoshExpansion}
\cos\Big(2k\cdot \arccosh\left[\cosh(\beta/2)-z\right]\Big)
=\\
\cos(k\beta)+
2k\sinh(2\pi k)
\int_0^{\infty}
\sum_{l=1}^{\infty} \frac{\Gamma\Big(\frac{l}{2}\pm i k \pm i w\Big)}{2\pi^2 \Gamma(l)}\cdot \frac{(2 z)^l}{l!} %
\cos(w\beta)\d w,
\end{multline}
where the right-hand side converges absolutely.
Thus,
\begin{multline}\label{eq_prp_bound_exp_moment_gener_func_LDP}
\int_{\R_+}\sum_{l=1}^{\infty}\frac{\Gamma\big(\frac{l}{2} \pm i k_1 \pm i k_2\big)}{2\pi^2\, \Gamma(l)}\cdot\frac{\lambda^l }{l!}\cdot \exp\left\{-\frac{t \sigma^2}{2}\cdot k^2_1\right\}\sinh(2\pi k_1)\, 2 k_1 \d k_1\\
=
\frac{2}{\pi}\int_{\R_+}\int_{\R_+} \exp\left\{-\frac{t\sigma^2}{2}\cdot k^2_1\right\}
\Bigg(\cos\Big(2k_1\, \arccosh\left[\cosh\left(\tfrac{\alpha}{2}\right)-\tfrac{\lambda}{2}\right]\Big)- \cos(k_1\alpha)\Bigg)
\cos\left(k_2 \alpha\right) \d \alpha \d k_1 .
\end{multline}
Notice that by Statement~\ref{stmArccoshTaylor_LDP}
\begin{equation}
\left|\cos\Big(2k_1\, \arccosh\left[\cosh\left(\tfrac{\alpha}{2}\right)-\tfrac{\lambda}{2}\right]\Big)- \cos(k_1\alpha)\right|
=
O\left(e^{10k_1-\alpha/10}\right),
\end{equation}
and so the integral in the right-hand side of \eqref{eq_prp_bound_exp_moment_gener_func_LDP} converges absolutely.
We calculate the integral in $k_1$ first,
\begin{equation}
\frac{2}{\pi}
\int_{\R_+}\int_{\R_+} \exp\left\{-\frac{t\sigma^2}{2}\cdot k^2_1\right\}
\cos(k_1\alpha)
\d k_1
\cos\left(k_2 \alpha\right) \d \alpha  =  \exp\left\{-\frac{t\sigma^2}{2}\cdot k^2_2\right\},
\end{equation}
and
\begin{multline}
\frac{2}{\pi}\int_{\R_+}\int_{\R_+} \exp\left\{-\frac{t\sigma^2}{2}\cdot k^2_1\right\}\cos\left(2k_1\, \arccosh\left[\cosh\left(\frac{\alpha}{2}\right)-\frac{\lambda}{2}\right]\right)
\d k_1 \cos\left(k_2 \alpha\right) \d \alpha\\
=
\frac{\sqrt{2}}{\sqrt{\pi t}\sigma}
\int_{\R_+}
\exp\left\{ -\frac{2}{t\sigma^2}
\arccosh^2\left[\cosh\left(\frac{\alpha}{2}\right)-\frac{\lambda}{2}\right]\right\}\cos\left(k_2 \alpha\right) \d \alpha\\
\leq
\frac{\sqrt{2}}{\sqrt{\pi t}\sigma}
\int_{\R_+}
\exp\left\{ -\frac{2}{t\sigma^2}
\arccosh^2\left[\cosh\left(\frac{\alpha}{2}\right)-\frac{\lambda}{2}\right]\right\}\d \alpha.
\end{multline}
Using Statement~\ref{stmArccoshTaylor_LDP}, we get that for some $C_1>0$
\begin{equation}
\int_{\R_+}
\exp\left\{ -\frac{2}{t\sigma^2}
\arccosh^2\left[\cosh\left(\frac{\alpha}{2}\right)-\frac{\lambda}{2}\right]\right\}\d \alpha
\leq
\int_{\R_+}
\exp\left\{ -\frac{\alpha^2}{2t\sigma^2}
+\frac{2\lambda}{t\sigma^2} + \frac{C_1|\lambda|(\alpha^2+ |\lambda|)}{2t\sigma^2}
\right\}\d \alpha.
\end{equation}
Therefore,
\begin{multline}
\exp\left\{-\frac{t\sigma^2}{2}\cdot k^2_2\right\}+
\int_{\R_+}\sum_{l=1}^{\infty}\frac{\Gamma\big(\frac{l}{2} \pm i k_1 \pm i k_2\big)}{2\pi^2\, \Gamma(l)}\cdot\frac{\lambda^l }{l!}\cdot \exp\left\{-\frac{t \sigma^2}{2}\cdot k^2_1\right\}\sinh(2\pi k_1)\, 2 k_1 \d k_1 \\
 \leq
 \frac{\sqrt{2}}{\sqrt{\pi t}\sigma}
 \int_{\R_+}
\exp\left\{ -\frac{\alpha^2}{2t\sigma^2}
+\frac{2\lambda}{t\sigma^2} + \frac{C_1|\lambda|(\alpha^2+ |\lambda|)}{2t\sigma^2}
\right\}\d \alpha.
\end{multline}

Hence, for $\lambda$ with $|\lambda|<\min\{C_1^{-1}/2, 1\}$, 
\begin{multline}
\int \exp\left\{\frac{2\lambda}{\sigma^2} \cdot
\obs{\phi}{0}{t} 
\right\}\d \FMeas{\sigma^2}(\phi) \\
\leq
\frac{\sqrt{2}}{\sqrt{\pi t}\sigma}
\int_{\R_+^2}
\exp\left\{ -\frac{\alpha^2}{2t\sigma^2}
+\frac{2\lambda}{t\sigma^2} 
+ \frac{C_1|\lambda|(\alpha^2+ |\lambda|)}{2t\sigma^2}
\right\}
\exp\left(-\frac{(1-t)\sigma^2}{2}\cdot k_2^2\right)
\sinh(2\pi k_2)\, 2k_2 \d k_2
\d \alpha.
\end{multline}
which finishes the proof.

\end{proof}

\bibliography{Schwarzian}
\end{document}